\documentclass[10 pt]{amsart}

\usepackage{tikz}
\usetikzlibrary{cd,decorations.pathmorphing}
\usepackage{mathtools}
\mathtoolsset{showonlyrefs} 
\usepackage[alphabetic]{amsrefs}

\newtheorem{thm}{Theorem}[section]
\newtheorem{lem}[thm]{Lemma}

\newtheorem{prop}[thm]{Proposition}
\newtheorem{lemma}[thm]{Lemma}

\theoremstyle{definition}
\newtheorem{defn}[thm]{Definition}
\newtheorem{definition}[thm]{Definition}
\newtheorem{ex}[thm]{Example}
\newtheorem{example}[thm]{Example}

\newtheorem*{notnonly}{Notation}

\theoremstyle{remark}  
\newtheorem{rem}[thm]{Remark}
\newtheorem{remark}[thm]{Remark}

\numberwithin{equation}{section}

\newcommand{\secref}[1]{Section~\textup{\ref{#1}}}

\newcommand{\defnref}[1]{Definition~\textup{\ref{#1}}}

\newcommand\id{\operatorname{id}}
\newcommand{\inv}{^{-1}}
\newcommand\set[1]{\{\,#1\,\}}
\newcommand\go{G^{(0)}}
\newcommand\ho{H^{(0)}}

\newcommand\lambdah{\lambda_{H}}
\newcommand\lambdag{\lambda_{G}}

\newcommand\tensor\otimes

\newcommand\sd[2]{S(#1,#2)}
\newcommand\ac[2]{#1\sdp #2}

\newcommand\sdp{\rtimes}
\newcommand\hsdg{\sd HG}

\newcommand\cca{\mathcal{C}_{c}(G*_{s}H,\AA)}

\newcommand\half{\frac12}

\newcommand\grho{\ho*_{r}G}

\newcommand\nugrho{\nu_{\grho}}

\newcommand\End{\operatorname{End}}
\newcommand{\BB}{\mathcal B}
\newcommand{\righttext}[1]{\quad\text{#1 }}

\renewcommand{\AA}{\mathcal A}
\newcommand{\EE}{\mathcal E}
\newcommand{\HH}{\mathcal H}

\newcommand{\KK}{\mathcal K}
\newcommand{\R}{\mathbb R}
\newcommand{\C}{\mathbb C}

\let\ipscriptstyle=\scriptscriptstyle
\def\lipsqueeze{{\mskip -3.0mu}}
\def\ripsqueeze{{\mskip -3.0mu}}
\def\ipcomma{\nobreak\mathrel{,}\nobreak}
\newbox\ipstrutbox
\setbox\ipstrutbox=\hbox{\vrule height8.5pt
width 0pt}
\def\ipstrut{\copy\ipstrutbox}
\def\lip#1<#2,#3>{\mathopen{\relax_{\ipstrut\ipscriptstyle{
#1}}\lipsqueeze
\langle} #2\ipcomma #3 \rangle}
\def\blip#1<#2,#3>{\mathopen{\relax_{\ipstrut
\ipscriptstyle{ #1}}\lipsqueeze\bigl\langle} #2\ipcomma #3 \bigr\rangle}
\def\rip#1<#2,#3>{\langle #2\ipcomma #3
\rangle_{\ripsqueeze\ipstrut\ipscriptstyle{#1}}}
\def\brip#1<#2,#3>{\bigl\langle #2\ipcomma #3
\bigr\rangle_{\ripsqueeze\ipstrut\ipscriptstyle{#1}}}
\def\angsqueeze{\mskip -6mu}
\def\smangsqueeze{\mskip -3.7mu}
\def\trip#1<#2,#3>{\langle\smangsqueeze\langle #2\ipcomma #3
\rangle\smangsqueeze\rangle_{\ripsqueeze\ipstrut\ipscriptstyle{#1}}}
\def\btrip#1<#2,#3>{\bigl\langle\angsqueeze\bigl\langle #2\ipcomma
#3
\bigr\rangle
\angsqueeze\bigr\rangle_{\ripsqueeze\ipstrut\ipscriptstyle{#1}}}
\def\tlip#1<#2,#3>{\mathopen{\relax_{\ipstrut\ipscriptstyle{
#1}}\lipsqueeze \langle\smangsqueeze\langle} #2\ipcomma #3
\rangle\smangsqueeze\rangle}
\def\btlip#1<#2,#3>{\mathopen{\relax_{\ipstrut\ipscriptstyle{
#1}}\lipsqueeze
\bigl\langle\angsqueeze\bigl\langle} #2\ipcomma #3
\bigr\rangle\angsqueeze\bigr\rangle}

\def\ip(#1|#2){(#1\mid #2)}
\def\bip(#1|#2){\bigl(#1 \mid #2\bigr)}
\def\Bip(#1|#2){\Bigl( #1 \bigm| #2 \Bigr)}
\def\Lip<#1,#2>{\lip\scriptstyle\star<#1,#2>}
\def\Rip<#1,#2>{\rip\scriptstyle\star<#1,#2>}

\newcommand\cs{\ensuremath{C^{*}}}
\newcommand{\ib}{im\-prim\-i\-tiv\-ity bi\-mod\-u\-le}

\newcommand{\sme}{\,\mathord{\mathop{\text{--}}\nolimits_{\relax}}\,}

\allowdisplaybreaks

\usepackage[shortlabels]{enumitem}

\newcommand{\lchs}{locally compact Hausdorff space}
\newcommand{\lcg}{locally compact groupoid}

\newcommand{\units}{^{(0)}}

\usepackage{amssymb}

\begin{document}
\title[Groupoid semidirect product Fell bundles I]{Groupoid semidirect
  product Fell bundles I ---\\ 
Actions by isomorphisms}

\author[Hall]{Lucas Hall}
\address{School of Mathematical and Statistical Sciences
\\Arizona State University
\\Tempe, Arizona 85287}
\email{lhall10@asu.edu}

\author[Kaliszewski]{S. Kaliszewski}
\address{School of Mathematical and Statistical Sciences
\\Arizona State University
\\Tempe, Arizona 85287}
\email{kaliszewski@asu.edu}

\author[Quigg]{John Quigg}
\address{School of Mathematical and Statistical Sciences
\\Arizona State University
\\Tempe, Arizona 85287}
\email{quigg@asu.edu}

\author[Williams]{Dana P. Williams}
\address{Department of Mathematics
\\Dartmouth College
\\Hanover, New Hampshire 03755}
\email{dana@math.dartmouth.edu}



\date{May 4, 2021}

\begin{abstract}
  Given an action of a groupoid by isomorphisms
  on a Fell bundle (over another groupoid), we
  form a semidirect-product 
  Fell bundle,
  and prove that its $C^*$-algebra is isomorphic to a crossed product.
\end{abstract}
\maketitle

\section{Introduction}

Groupoid \cs-algebras are a natural generalization of transformation
group \cs-algebras first studied by Effros and Hahn back in the 1960s
\cite{effhah:mams67}.  Since then the construction of \cs-algebras
from groupoids has gone through myriad generalizations including
groupoid crossed products modeled on crossed-product \cs-algebras and
many others that are all subsumed by Fell-bundle \cs-algebras.  In
\cite{kmqw1, kmqw2}, Muhly and the last three named authors consider
actions of a group $G$ on a Fell bundle $p:\AA\to H$ by automorphisms
and analyze the consequences for the group semidirect-product
$\AA\rtimes G\to H\rtimes G$. When the group action is specialized,
the authors deduce an equivalence theorem \cite{kmqw1}*{Theorem 3.1}
which unifies many imprimitivity theorems under the context of Fell
bundles and their semidirect products.

As suggested in \cite{kmqw1}, we want to extend some of the results
there to actions by groupoids.  In this note, we take up this task.
We let a groupoid $G$ act on another groupoid $H$ by groupoid
isomorphisms.  Then we can build a semidirect-product groupoid
$\sd HG$.  (We have purposely avoided the usual group notation
  for semidirect products,
  $H\rtimes G$, as this coincides with the usual notation for a
  transformation groupoid.) Given a Fell bundle $p:\AA\to H$ and an
action of $G$ on $\AA$ by isomorphisms covering the given action on
$H$, we can form a Fell bundle $q:\sd {\AA}G\to\sd HG$.  Our main
result is that the corresponding Fell-bundle \cs-algebra
$\cs(\sd HG,\sd {\AA}G)$ is isomorphic to a groupoid crossed product
$\cs(H,\AA)\rtimes_{\alpha}G$ for an action $\alpha$.  This leaves off
the task of extending the equivalence theorem of \cite{kmqw1} from
group to groupoid actions, which we accomplish in the forthcoming
article \cite{hkqwstab}.

We begin in \secref{sec:prelim} by recording our conventions for
locally compact groupoids, groupoid actions on spaces, action
groupoids, upper-semicontinuous Banach bundles, Fell bundles
$p:\AA\to G$ over a groupoid $G$, and pull-back Fell bundles.

In \secref{sec:bundl-prop-acti} we record our conventions for bundles
$(T,p,B)$, morphisms of bundles, and pull-backs.  In \secref{sec:act
  on bundle} we record our conventions for an action of a groupoid $G$
on a bundle $p:T\to B$, and in \secref{sec:actions} we apply this to
actions on Fell bundles over groupoids.

In \secref{sec:semidirect} we develop a notion of a semidirect product
$\sd HG$ of group\-oids $H$ and $G$ when $G$ acts on $H$ by
isomorphisms.  We then promote this to semidirect-product Fell bundles
where $G$ acts on a Fell bundle $p:\AA\to H$ by isomorphisms.

Finally, in \secref{sec:main} we prove that if $G$ acts on a Fell
bundle $p:\AA\to H$ by isomorphisms, then, under a mild assumption
involving Haar systems, the $C^*$-algebra $\cs(\sd HG,\sd \AA G)$ of
the semidirect-product Fell bundle is isomorphic to the crossed
product by an appropriate action of $G$ on the Fell-bundle algebra
$\cs(H,\AA)$.

A alternate version of Theorem~\ref{thm-cross-prod} could be
  derived from the unpublished preprint \cite{bm:xx16} of Buss and
  Meyer.  Their ``iterated-crossed-product decomposition'' result
  (Theorem~6.1 in their paper) would exhibit $\cs(\sd HG,\sd \AA G)$
  as the \cs-algebra of an abstract Fell bundle over $G$.  Because of
  the general nature of their result,
  considerable work would be required to show that their Fell
  bundle \cs-algebra is isomorphic to the groupoid crossed product
  $\cs(H,\AA)\rtimes_{\alpha}G$.  In particular, an extra hypotheses on
  the Haar systems is required to build the isomorphisms for the
  action $\alpha$ in the dynamical system.  Obtaining the crossed
  product decomposition in Theorem~\ref{thm-cross-prod} gives the
  advantage of more concrete realization of $\cs(\sd HG,\sd \AA G)$ as
  a crossed product rather than a more general Fell bundle
  \cs-algebra.  For example, crossed products can
  be studied using covariant representations as
  outlined in \cite{muhwil:nyjm}*{Section~7}.  Furthermore, our
  presentation is self-contained and does not require untangling, or
  reference to,  the
  impressive machine from \cite{bm:xx16}.

We also mention the work of Duwenig and Li involving Zappa-Sz\'ep
products for Fell bundles over \'etale groupoids
\cite{dlzappa}. Following Brownlowe \emph{et al}, they allow for two
groupoids to act on each other in a way that generalizes our semidirect
product. They then consider how the Zappa-Sz\'ep product of groupoids
lifts to an associated product for Fell bundles over an \'etale
groupoid. They thus recover and generalize our result, as long as one
considers the significantly restrictive case of \'etale groupoids. In
view of the present article, there is promise that the Duwenig-Li
construction generalizes to non-\'etale groupoids, but we do not
pursue this here.

\section{Preliminaries}\label{sec:prelim}

Throughout, $G$ and $H$ will be a second countable, locally compact Hausdorff
groupoids with Haar systems $\lambdag=\set{\lambdag^{u}}_{u\in \go}$
and $\lambdah=\set{\lambdah^{v}}_{v\in\ho}$, respectively.
In fact, when not obviously contradicted by context, the term
``groupoid'' means a locally compact Hausdorff groupoid with open
range and source maps.

\subsection{Open Maps}  Since open maps often play a key role in
the theory, it is useful to have a criterion for establishing that a
given map is open.   The following is called ``Fell's
Criterion'' in
\cite{wil:toolkit}*{Proposition~1.1} and comes from
\cite{felldoran}*{Proposition~II.13.2}. 

\begin{lem}[Fell's Criterion] \label{lem-fc}
  Suppose that $f:X\to Y$ is a surjection.  Then $f$ is open if and
  only if given a net $\set{y_{j}}_{i\in I}$ converging to $f(x)$ in
  $Y$, then there is a subnet $\set{y_{j}}_{j\in J}$ and a net
  $\set{x_{j}}_{j\in J}$ such that $x_{j}\to x$ and $f(x_{j})=y_{j}$.
\end{lem}

We often use Fell's Criterion to ``lift'' nets as follows.  Suppose that $f:X\to
Y$ is an open surjection and $y_{i}\to f(x)$.  Then we can pass to a
subnet $\set{y_{i_{j}}}$ and find $x_{j}\to x$ with
$f(x_{j})=y_{i_{j}}$.  Typically, we dispense with the distracting
subnet notation and replace the original net with the subnet.  We
summarize this by saying that ``we can replace $\set{y_{i}}$ with a
subnet, relabel, and assume that are $x_{i}\to x$ with
$f(x_{i})=y_{i}$''.

Nets and subnets are discussed briefly at the end of
\cite{danacrossed}*{\S1.1}. 

\subsection{Groupoid Actions}
\label{sec:groupoid-actions}

Groupoid actions are treated in detail in \cite[\S2.1]{wil:toolkit}.
In particular, in order for $G$ to act on the left of a space $T$, we
need a continuous \emph{moment map}
$\rho_{T}:T\to\go$.  Then we require a a continuous
map from $G*T=\set{(x,t)\in G\times T:s_{G}(x)=\rho_{T}(t)}$ to $T$ such that
the usual axioms hold---see \cite[Definition~2.1]{wil:toolkit}.  This
means that $\rho_{T}(t)\cdot t=t$ for all $t\in T$,
$\rho_{T}(x\cdot t)=r_{G}(x)$, and $x\cdot (y\cdot t)=(xy)\cdot t$
whenever $s_{G}(y)=\rho_{T}(t)$ and $(x,y)\in G^{(2)}$.  In the
literature, it is often assumed that the moment map, $\rho_{T}$, is
open as well as continuous.  Here we will only make that assumption
when necessary. Except when it would be
confusing otherwise, it is common to delete the subscripts from
expressions such as $s_{G}(x)=\rho_{T}(t)$ and simply write
$s(x)=\rho(t)$ assuming that the meaning is clear from context.
Right actions are defined similarly except that the moment map is
usually denoted by $\sigma_{T}$ and then $x\cdot t$ is defined when
$r_{G}(x)=\sigma_{T}(t)$.

\begin{notnonly}
  We
employ the following standard notation for a $G$-space $T$ with moment
map $\rho_{T}:T\to\go$, $N\subset G$, and $S\subset T$:
\begin{enumerate}[(a),nosep]
\item
$N*S=\set{(x,t)\in N\times S:s(x)=\rho(t)}$;
\item
$N\cdot S=\set{x\cdot t:(x,t)\in N*S}$;
\item
$S*T
=\set{(t,u)\in S\times T:u\in G\cdot t}$.
\end{enumerate}
\end{notnonly}

\begin{defn}\label{def-act-grp}
Let a \lcg\ $G$ act on the left of a \lchs\ $T$.
Then the \emph{action groupoid}, $\ac GT$, 
is the space $G*T=\set{(x,t)\in G\times T:s(x)=\rho(t)}$ 
with operations given by
\[
(x,t)(y,s)=(xy,s)\text{\enspace if $t=y\cdot s$}\quad\text{and}\quad
(x,t)^{-1}=(x^{-1},x\cdot t).
\]
\end{defn}

Notice that the range and source maps are given by
\begin{align}
r(x,t)=(r(x),x\cdot t)\quad\text{and}\quad
s(x,t)=(s(x),t).
\end{align}
Then the unit space is
\[
(\ac GT)\units=\set{(u,t)\in G\units\times T:u=\rho(t)}.
\]
Since the coordinate projection
\[
\pi_2:(\ac GT)\units\to T
\]
is a homeomorphism, it is common practice to identify $(\ac GT)\units$
with $T$.  The action groupoid $\ac GT$ has open range and source maps
when $G$ does. 

\begin{remark}\label{rem-other}
  [Action Groupoids in the Literature] In the literature, the
  formulation of an action groupoids---also called transformation
  groupoids---is not consistent.  If $T$ is a left $G$-space, then in
  \cite{wil:toolkit}*{Definition~2.5} the action groupoid $A(G,T)$ is
  the set $\set{(t,x,t')\in T\times G\times T:t=x\cdot t'}$ of triples
  equipped with the relative product topology and with the natural
  groupoid operations: $(t,x,t')(t',y,t'')=(t,xy,t'')$ and
  $(t,x,t')^{-1}= (t',x^{-1},t)$.  Using triples makes the definition
  both natural and transparent.  It is also redundant.  As a result,
  it is common practice to denote elements of $A(G,T)$ more compactly
  as pairs in $G\times T$.  Our definition of $\ac GT$ above is pulled
  back from the homeomorphism $(x,t)\mapsto (x\cdot t, x, t)$ of $G*T$
  onto $A(G,T)$.  However, one can also consider the set
  $\set{(x,t):r(x)=\rho(t)}$ and pull-back a groupoid structure via
  the map $(x,t)\mapsto (x,t,x^{-1}\cdot t)$.  This groupoid structure
  has the advantage that if $G$ is a group, then we arrive at the
  usual formulas for transformation group algebras.  The convention
  here has the advantage that it is the same as used in
  \cite{kmqw2}*{Appendix~A.1}.  However, it differs, for example, from
  that used in \cite{wil:toolkit} or \cite{AR}. Of course, there are
  similar considerations for right actions.
\end{remark}

\subsection{Banach Bundles}
\begin{definition}
  \label{def-usc-bundle} An (upper-semicontinuous) \emph{Banach bundle} over a
  space $X$ is a topological space $\AA$ together with a continuous
  open surjection $p:\AA\to X$ and Banach space structures on the
  fibres $A(x)=p^{-1}(x)$ such that the following axioms
  hold.
  \begin{enumerate}[({B}1)]
  \item $a\mapsto \|a\|$ is upper-semicontinuous from $\AA$ to
    $\R^{+}$.
  \item $(a,b)\mapsto a+b$ is continuous from
    $\AA*\AA=\set{(a,b):p(a)=p(b)}$ to $\AA$.
  \item $(\lambda,a)\mapsto \lambda a$ is continuous from
    $\C\times \AA$ to $\AA$.
  \item If $\set{a_{i}}$ is a net in $\AA$ such that $p(a_{i})\to x$
    and $\|a_{i}\|\to 0$, then $a_{i}\to 0_{x}$ where $0_{x}$ is the
    zero element in $A(x)$.
  \end{enumerate}
  We use $A(x)$ to denote the fibre over $x$---rather
    than $\AA_{x}$, say---to emphasize that the fibre comes equipped
    with its own fixed Banach space structure.
\end{definition}

In in addition each fibre $A(x)$ is a \cs-algebra and $(a,b)\mapsto
ab$ and $a\mapsto a^{*}$ are continuous on $\AA*\AA$ and $\AA$,
respectively, then we call $\AA$ a (upper-semicontinuous) \cs-bundle.
If axiom~(B1) is replaced by ``$a\mapsto \|a\|$ is continuous'', then
we call $\AA$ a continuous Banach bundle or a continuous \cs-bundle.
Normally, we drop the adjective ``upper-semicontinuous'' and add
``continuous'' only when specializing to that case.

We use \S\S13--14 of
\cite{felldoran}*{Chap.~II} as a good reference for continuous Banach bundles.  For the general case, see
\cite{mw:fell}*{Appendix~A}, and for the \cs-case,
\cite{danacrossed}*{Appendix~C}.  A primary motivation for working
with general Banach bundles rather than continuous ones is that a
$C_{0}(X)$-algebra $A$ is always $C_{0}(X)$-isomorphic to
$\Gamma_{0}(X,\AA)$ for an appropriate (upper-semicontinuous) Banach
bundle $\AA$ \cite{danacrossed}*{Theorem~C.26}.

\subsection{Fell Bundles}

For Fell bundles we will refer to \cite{mw:fell}.
\begin{definition}[\cite{mw:fell}*{Definition~1.1}]
  \label{def-fell-bundle} Suppose that $p:\AA\to G$ is a separable
  Banach bundle over a second countable locally
  compact Hausdorff groupoid $G$.  Let
  $\AA^{(2)}=\set{(a,b)\in \AA\times\AA: (p(a),p(b))\in G^{(2)}}$.
  Then we call $\AA$ a \emph{Fell bundle} if there is a continuous,
  bilinear, associative multiplication map $(a,b)\mapsto ab$ from
  $\AA^{(2)}\to \AA$ and a continuous involution $a\mapsto a^{*}$ from
  $\AA\to\AA$ such that
  \begin{enumerate}[\rm ({FB}1)]
  \item\label{FB1} $p(ab)=p(a)p(b)$,
  \item\label{FB2} $p(a^{*})=p(a)^{-1}$,
  \item\label{FB3} $(ab)^{*}=b^{*}a^{*}$,
  \item\label{FB4} for each $u\in\go$, the Banach space fibre $A(u)$
    is a \cs-algebra with respect to the structure induced by the
    multiplication and involution on $\AA$, and
  \item\label{FB5} for each $x\in G$, the Banach space fibre $A(x)$ is
    an $A(r(x))\sme A(s(x))$-\ib\ when equipped with the actions
    determined by multiplication and the inner products
    \begin{equation}
      \label{eq:20}
      {}_{A(r(x))}\langle a,b\rangle=ab^{*}\quad\text{and}\quad
      \langle a,b\rangle_{A(s(x))}=a^{*}b.
    \end{equation}
  \end{enumerate}
\end{definition}

If $p:\AA\to H$ is a Fell bundle over a locally compact groupoid (and
we tacitly assume that our groupoids are Hausdorff), then to ease the
notational burden, for $a\in\AA$, we will write $r(a)$ in place of
$r_{H}(p(a))$ and similarly for $s(a)$.

\begin{remark}[Pull-Back Fell Bundles]
  \label{rem-pull-back-fell} Suppose that $p:\AA\to G$ is a Fell
  bundle and that $\phi:H\to G$ is a groupoid homomorphism.  Then the
  pull-back,
  \begin{equation}
    \label{eq:21}
    \phi^{*}\AA=\set{(h,a)\in H\times\AA:\phi(h)=p(a)}
  \end{equation}
  is a Fell bundle over $H$ with respect to $(h,a)(h',a')=(hh',aa')$
  and $(h,a)^{*}=(h^{-1},a^{*})$.
\end{remark}

\subsection{Fell Bundle Maps}

If $p:\AA\to H$ and $p':\BB\to G$ are Fell bundles.   Then a
\emph{Fell-bundle map} $\phi:\AA\to\BB$ is a bundle map
\begin{equation}
  \label{eq:17}
  \begin{tikzcd}
    \AA \arrow[d,"p",swap] \arrow[r,"\phi"] &\BB \arrow[d,"p'"] \\
    H\arrow[r,"\bar\phi",swap] & G
  \end{tikzcd}
\end{equation}
such that the induced map $\bar\phi$ is a groupoid homomorphism and
such that $\phi$ preserves multiplication and involution.   That is,
$(a,a')\in\AA^{(2)}$ implies $(\phi(a),\phi(a'))\in\BB^{(2)}$ and then
$\phi(ab)=\phi(a)\phi(b)$, and of course, we also require
$\phi(a^{*})=\phi(a)^{*}$.
Note that if $v\in\ho$, then $\phi$ induces a $*$-homomorphism of
$A(v)$ to $B(\bar\phi(v))$.   Then for all $a\in\AA$, we have
\begin{align}
  \label{eq:18}
  \|\phi(a)\|^{2}=\|\phi(a^{*}a)\|\le \|a^{*}a\|=\|a\|^{2},
\end{align}
and $\phi$ must be norm decreasing.  Therefore if $\phi$ is a
\emph{Fell-bundle isomorphism}---that is, if $\bar\phi$ is a groupoid
isomorphism and $\phi$ is a homeomorphism---then $\phi$ is
isometric.

\subsection{Fell Bundle \cs-Algebras}
If $p:\AA\to G$ is a Fell bundle and $G$ has a Haar system, then
$\Gamma_{c}(G,\AA)$ has a natural $*$-algebra structure:
\begin{equation}
  \label{eq:6}
  f*g(x)=\int_{G}f(y)g(y^{-1}x)\,d\lambdag^{r(x)}(y)\quad\text{and}
  \quad f^{*}(x)=f(x^{-1})^{*}.
\end{equation}
The corresponding \emph{Fell-bundle \cs-algebra} is the completion
$\cs(G,\AA)$ with respect to the universal norm on $\Gamma_{c}(G,\AA)$
induced by suitably bounded $*$-representations of $\Gamma_{c}(G,\AA)$
on Hilbert space as in \cite{mw:fell}*{\S1}.

\subsection{Groupoid Crossed Products}

In our main theorem, we will want to exhibit a Fell-bundle \cs-algebra
as the \cs-algebra of a groupoid crossed product.   We will refer to
\cite{muhwil:nyjm} for the basic theory, although Goehle's thesis
\cite{goe:thesis} is very good resource.   We recall the basics here.
\begin{definition}
  Suppose that $G$ is a locally compact groupoid and that $A$ is a
  $C_{0}(\go)$-algebra such that $A=\Gamma_{0}(\go,\AA)$ for a
  \cs-bundle over $\go$.  An \emph{action} $\alpha$ of $G$ on $A$ by
  $*$-isomorphisms  is a family $\set{\alpha_{x}}_{x\in G}$
  such that
  \begin{enumerate}
  \item for each $x\in G$, $\alpha_{x}:A(s(x))\to
    A(r(x))$ is an isomorphism,
  \item for all $(x,y)\in G^{(2)}$,
    $\alpha_{xy}=\alpha_{x} \circ \alpha_{y}$, and
  \item $x\cdot a=\alpha_{x}(a)$ defines a continuous action
    of $G$ on $\AA$.
  \end{enumerate}
The triple $(\AA,G,\alpha)$ is called a \emph{groupoid dynamical system}.
\end{definition}

Classically, the groupoid crossed product $A\rtimes_{\alpha}G$ is
built out of the sections $\Gamma_{c}(G,r^{*}\AA)$.   In the sequel,
we will use the observation from \cite[\S2]{mw:fell} that we can
realize $A\rtimes_{\alpha}G$ as the Fell-bundle \cs-algebra
$\cs(G,\BB)$ where $\BB=r^{*}\AA=\set{(a,x)\in\AA\times G:p(a)=r(x)}$
with multiplication given by $(a,x)(b,y)=(a\alpha_{x}(b),xy)$ and
involution by $(a,x)^{*}=(\alpha_{x}^{-1}(a),x^{-1})$.

\section{Bundles and Actions}
\label{sec:bundl-prop-acti}

Here we will use the term \emph{bundle} for a triple $(T,p,B)$
consisting of a continuous \emph{open} surjection
$p:T\to B$.
  We call $T$ the
\emph{total space} and $B$ the \emph{base space} of the bundle.
We are using a definition of ``bundle'' that is
  suitable for the context in which we work.  In \cite{husemoller},
  a bundle $p:T\to B$ is merely a continuous map.  We add
  the requirements that $p$ be open and surjective, but we do not
  require local triviality in any form.

If $T$ is a left $G$-space, then we call a bundle $(T,p,B)$ a
\emph{$G$-bundle} if $p(t)=p(t')$ if and only if $t$ and $t'$ are in
the same orbit.  Alternatively, $p:T\to B$ is a $G$-bundle if and only
if there is a homeomorphism $\bar p$ of $B$ with the orbit space
$G\backslash T$ such that the diagram
\begin{equation}
  \label{eq:1}
  \begin{tikzcd}
    &T\arrow[dl,"q",swap] \arrow[dr,"p"] \\
    G\backslash T\arrow[rr,"\bar p",swap]&& B
  \end{tikzcd}
\end{equation}
commutes.

Let $(T,p,B)$ and $(Y,q,C)$ be $G$-bundles.  Then a \emph{$G$-bundle
  morphism} $(f,g):(T,p,B)\to (Y,q,C)$ is a pair of continuous maps
such that $f$ is $G$-equivariant and such that the diagram
\begin{equation}
  \label{eq:2}
  \begin{tikzcd}
    T\arrow[r,"f"] \arrow[d,"p",swap] & Y \arrow[d,"q"] \\
    B\arrow[r,"g",swap] & C
  \end{tikzcd}
\end{equation}
commutes.  If $B=C$ and $g=\id_B$, we say $f:T\to Y$ is a $G$-bundle
morphism \emph{over $B$}.  In any event, we say that $(f,g)$ is an
\emph{isomorphism} if $f$ is a homeomorphism.

If $T$ and $U$ are $G$-spaces, then the fibred product
$T*_{r}U=\set{(t,u):\rho_{T}(t)=\rho_{U}(u)}$ admits a \emph{diagonal}
$G$-action: $x\cdot (t,u)=(x\cdot t, x\cdot u)$ with
$\rho(t,u)=\rho_{T}(t)=\rho_{U}(u)$.

\begin{lem}
  [Diagonal Actions] \label{diagonal} If $T$ and $U$ are $G$-spaces,
  then 
  the moment map 
  for the diagonal action of $G$
  on $T*_{r}U$
  is open if both $\rho_{T}$ and $\rho_{U}$ are.
\end{lem}
\begin{proof}
  We use Fell's Criterion
  (Lemma~\ref{lem-fc}): suppose $v_{i}\to \rho_{T*U}(t,u)$.  Since
  $\rho_{T}$ is open, we can pass to subnet, relabel, and assume that
  there are $t_{i}\to t$ with $\rho_{T}(t_{i})=v_{i}$.  Since $\rho_{U}$ is also
  open, we can pass to another subnet, relabel, and assume that there
  are $u_{i}\to u$ such that $\rho_{U}(u_{i})=v_{i}$.  But then
  $(t_{i},u_{i})\to (t,u)$.
\end{proof}

If $(T,p,B)$ is a bundle and $f:C\to B$ is continuous, then we can
form the \emph{pull-back} $f^{*}T=C{}_{f}*_{p}T=\set{(c,t):f(c)=p(t)}$
so that the diagram
\begin{equation}
  \label{eq:5}
  \begin{tikzcd}
    f^{*}T=C*T \arrow[d,"\pi_{1}",swap] \arrow[r,"\pi_{2}"] & T
    \arrow[d,"p"] \\
    C \arrow[r,"f",swap]& B
  \end{tikzcd}
\end{equation}
commutes.

\begin{lem}
  [Pull-Backs] \label{pullback} Let $(T,p,B)$ be a bundle and
  $f:C\to B$ a continuous map.  Then $(f^{*}T,\pi_{1},C)$ is a bundle
  where $\pi_{1}$ is the projection onto the first factor.  If $T$ is
  a $G$-bundle, then so is $f^{*}T$ where the $G$-action is given by
  $x\cdot (c,t)=(c,x\cdot t)$.  Furthermore, $f^{*}T$ 
  has an open moment map whenever $T$ does.
\end{lem}
\begin{proof}
  To see that $\pi_{1}$ is open, we use Fell's Criterion.  Suppose
  that $c_{i}\to \pi_{1}(c,t)$.  Then $f(c_{i})\to f(c)=p(t)$.  Since
  $p$ is open, we can pass to subnet, relabel, and assume that there
  are $t_{i}\to t$ in $T$ with $p(t_{i})=f(c_{i})$.  But then
  $(c_{i},t_{i})\to (c,t)$ in $f^{*}(T)$.  Hence the pull-back is
  always a bundle over $C$.

  Now suppose that $T$ is a $G$-bundle.  If $(c,t)$ and $(c,t')$ both
  belong to $f^{*}(T)$, then $p(t)=f(c)=p(t')$ and there is a $x\in G$
  such that $t'=x\cdot t$.  Then $(c,t')=x \cdot (c,t)$.  Since the
  converse is clear, $\pi_{1}:f^{*}T\to C$ is a $G$-bundle.

  Now suppose that $\rho_{T}:T\to\go$ is open.  To see that
  $\rho_{f^{*}T}$ is open, we suppose that
  $v_{i}\to v=r(c,t)=\rho_{T}(t)$.  Since $\rho_{T}$ is open, after
  passing to a subnet and relabeling, we can assume that there are
  $t_{i}\to t$ with $\rho_{T}(t_{i})=v_{i}$.  Then
  $p(t_{i})\to p(t)=f(c)$ in $B$.  Since we are now assuming that $f$
  is open, we can pass to a subnet, relabel, and assume that there are
  $c_{i}\to c$ in $C$ such that $f(c_{i})=p(t_{i})$.  Then
  $(c_{i},t_{i})\to (c,t)$ as required.
\end{proof}

Note that if $(T,p,B)$ is a $G$-bundle and $f:C\to B$ is a continuous
map, then $(\pi_{2},f):(f^{*}T,\pi_{1},C)\to (T,p,B)$ is bundle
morphism.

\section{Actions on bundles}\label{sec:act on bundle}

Everything in this section is surely standard, although we could not
find it in the literature in the precise context we need; we include
the details for convenient reference. 

  \begin{defn}\label{act on bundle}
  Let $G$ be a locally compact groupoid, and let $p:T\to B$ be a bundle.
  We say that \emph{$G$ acts on the bundle $p:T\to B$} if $G$ acts on
  both the total space $T$ and the base space $B$ and the bundle map
  is equivariant.  In particular, $p$ intertwines the moment maps:
  $\rho_{B}\circ p=\rho_{T}$.
\end{defn}

\begin{rem}
Let $p:T\to B$ be a bundle and $C\subset B$ a closed subset,
and put $T|_C=p\inv(C)$.
Then we have a restricted bundle
\[
p|_C:T|_C\to C.
\]
The only non-obvious property to check is that the restriction $p|_C$
is open: let $c_i\to p(t)$ in $C$.  Since $p:T\to B$ is open, after
replacing by a subnet and relabeling we can find $t_i\in p\inv(c_i)$
such that $t_i\to t$ in $T$.  Then in particular $t_i\in p_u\inv(c_i)$
and $t_i\to t$ in $T|_C=p\inv(C)$.
\end{rem}

\begin{rem}
Let $G$ act on a bundle $p:T\to B$,
and
for $u\in G\units$ let $B_u=\rho_B\inv(u)$.
Then for all $x\in G$ we have a bundle isomorphism
\begin{equation}
  \label{eq:8}
  \begin{tikzcd}[column sep=2cm]
    T|_{B_{s(x)}} \arrow[d,"p",swap] \arrow[r,"t\mapsto x\cdot t"] &
    T|_{B_{r(x)}} \arrow[d,"p"] \\
    B_{s(x)}\arrow[r,"b\mapsto x\cdot b",swap] & B_{r(x)}.
  \end{tikzcd}
\end{equation}
In particular, the upper horizontal map is a homeomorphism between the
respective subsets of $T$.
\end{rem}

\section{Actions by isomorphisms}
\label{sec:actions}

In this section we adapt some of \cite[Appendix~A]{kmqw2} to groupoid actions.


\begin{defn}
A \emph{groupoid bundle} is a bundle $p:H\to B$ in which $B$ is a
locally compact space,
$H$ is a groupoid, and
each fibre $H_b:=p\inv(b)$ is a subgroupoid of $H$.
\end{defn}

\begin{rem}
  \label{rem-sat} If $p:H\to B$ is a groupoid bundle, then $h\in
  p^{-1}(b)$ implies $r(h)=hh^{-1}\in p^{-1}(b)$.  Thus $h\in H_{b}$
  if and only if $p(r(h))=b$.   Thus if $p_{0}=p|_{\ho}$, then
  $p_{0}\circ r=p_{0}\circ s$ and $H_{b}$ is
  the reduction $H|_{p_{0}^{-1}(b)}=p_{0}^{-1}(b)Hp_{b}^{-1}(b)$ and
  $H$ is a bundle of groupoids 
  as in \cite{wil:toolkit}*{Definition~1.16}.
\end{rem}

\begin{defn}\label{act on groupoid}
Let $H$ be a locally compact groupoid, and let $G$ act on the space $H$.
We say $G$ \emph{acts by isomorphisms} if the moment
map $\rho_{H}:H\to G\units$ is a groupoid bundle and
for each $x\in G$ the map
\[
t\mapsto x\cdot t:H_{s(x)}\to H_{r(x)}
\]
is a groupoid isomorphism.  
\end{defn}

\begin{remark}
  In the unpublished preprint \cite{bm:xx16}*{Definition~2.12}, 
  actions by isomorphisms are called ``classical actions''. 
\end{remark}

\begin{rem}\label{rem-g-ho}
  If $G$ acts on $H$ by isomorphisms, then $H^{(0)}$ is a
  $G$-invariant subset.   Moreover, $r_{H}(x\cdot h)=x\cdot
  r_{H}(h)$.  
\end{rem}

\begin{ex}\label{ex-g-act}
  Let $H=\set{(x,y,z)\in G\times G\times G: \text{$s(x)=r(y)$ and
      $z=yx$}}$ be the action groupoid for the
  right action of $G$ on itself, with operations given by $(x,y,z)(z,y',x')=(x,yy',x')$ and
  $(x,y,z)^{-1}= (z,y^{-1},x)$.  Then $G$ acts on the left
  of $H$: $w\cdot (x,y,xy)=(wx,y,wxy)$.  Let
  $\rho_{H}(x,y,z)=r(x)$.  Then
  $H_{u}=\set{(x,y,xy)\in H:r(x)=u}=G^{u}\ltimes G$ is easily seen to
  be a subgroupoid of $H$ and $\rho_{H}$ is a groupoid
  bundle.
  Furthermore,
  \begin{align}
    \label{eq:9}
    w\cdot \bigl((x,y,xy)(xy,z,xyz)\bigr)= w\cdot (x,yz,xyz)=(wx,yz,wxyz),
  \end{align}
and on the other hand,
\begin{align}
  \label{eq:10}
  \bigl(w\cdot (x,y,xy)\bigr)\bigl(w\cdot (xy,z,xyz)\bigr) =
  (wx,y,wxy)(wxy,z,wxyz)=(wx,yz,wxyz) .
\end{align}
Hence $G$ acts on $H$ by isomorphisms.  
\end{ex}

\begin{defn}
  [cf., \cite{kmqw1}*{Definition~6.3}] \label{def-invariant}Suppose
  that $H$ has a Haar 
  system $\lambdah=\set{\lambdah^{v}}_{v\in\ho}$.   Then we say that
  an action of $G$ on $H$ by isomorphisms is \emph{invariant} if
  \begin{equation}
    \label{eq:15}
    \int_{H} f(x\cdot h)\,d\lambdah^{v}(h) = \int_{H}
    f(h)\,d\lambdah^{x\cdot v}(h)
  \end{equation}
  for all $f\in C_{c}(H)$,
      $x\in G$, and $v\in\ho$.
\end{defn}

\begin{remark}
  \label{rem-equiv} While it might be more appropriate to describe
  actions satisfying \eqref{eq:15} as ``equivariant'', the term
  ``invariant'' 
  is more common in the literature (see
  \cite{wil:toolkit}*{Remark~3.12}).  In any event, if we view $H$ and
  $\ho$ as $G$-spaces, then the invariance of the
  $G$-action translates to
  saying that $\lambdah$ is a ``full invariant $r_{H}$-system'' as defined
  in \cite{wil:toolkit}*{\S3.2}.   
\end{remark}


Now we promote the preceding to Fell bundles.

\begin{definition}\label{def-fell-act-iso}
  Let a locally compact groupoid $G$ act on a Fell bundle $p:\AA\to H$
  in the sense of \defnref{act on bundle} with respective moment maps
  $\rho_\AA$ and $\rho_H$.  We say that \emph{$G$ acts on $\AA$ by
    isomorphisms} if
  \begin{enumerate}[(a)]
  \item $G$ acts on the groupoid $H$ by isomorphisms, and
  \item for each $x\in vGu$ the map
    \[
      a\mapsto x\cdot a:\AA|_{H_u}\to \AA|_{H_v}
    \]
    is a Fell-bundle isomorphism.
  \end{enumerate}
\end{definition}
Note that we have a restricted Fell bundle $p_u:\AA|_{H_u}\to H_u$ for
every $u\in G\units$.  

\begin{example}
  \label{ex-g-act-fell}
  As in Example~\ref{ex-g-act}, let $H$ be the action groupoid for the
  right action of $G$ on itself.  Let $p:\AA\to G$ be a Fell bundle.
  Then $\phi(x,y,xy)=y$ is a homomorphism and we can form the
  pull-back Fell bundle $\phi^{*}\AA=\set{(x,y,xy,a):y=p(a)}$.  Then
  $G$ acts on $\phi^{*}\AA$ by isomorphisms where
  $z\cdot (x,y,xy,a)=(zx,y,zxy,a)$ (which clearly covers the action of
  $G$ on $H$ by isomorphisms given in Example~\ref{ex-g-act}).
\end{example}

\section{Semidirect products}\label{sec:semidirect}

When $G$ and $H$ are groups and $G$ acts on $H$ via automorphisms,
then we can construct their semidirect-product $\sd HG$
containing $H$ as a normal subgroup. In this section we develop an analogue of this
for groupoids.

\begin{definition}\label{def-sd-product}
  Suppose that $G$ and $H$ are locally compact groupoids such that $G$
  acts on $H$ by isomorphisms.  Then the \emph{semidirect-product
    groupoid} is the set
  $\sd HG=\set{(h,x)\in H\times G:\rho(h)=r(x)}$, where we declare
  $(h,x)$ and $(k,y)$ to be composable if $s(h)=x\cdot r(k)$.  Then we
  define
  \begin{align}
    \label{eq:37}
    (h,x)(k,y)=(h(x\cdot k),xy)\quad\text{and}\quad
    (h,x)^{-1}=(x^{-1}\cdot h^{-1},x^{-1}).
  \end{align}
\end{definition}

If $s(h)=x\cdot r(k)$, then $s(h)=r(x)$ and 
$s(x)=\rho(r(k))=\rho(k)=r(y)$.  Hence $x\cdot k$ is defined,
$(h,x\cdot k)\in H^{(2)}$ and $(x,y)\in G^{(2)}$.  So the product in
\eqref{eq:37} makes sense.  Then routine computations---say following
\cite{wil:toolkit}*{Definition~1.2}---show that $\sd HG$ is a groupoid with
\begin{align}
  \label{eq:38}
  s(h,x)=(x^{-1}\cdot
          s_{H}(h),s_{G}(x))\quad\text{and}\quad 
  r(h,x)=(r_{H}(h),r_{G}(x)).
\end{align}
In particular,
\begin{equation}
  \label{eq:39}
  \sd HG^{(0)}=\set{(v,u)\in \ho\times\go:\rho_{H}(v)=u},
\end{equation}
and the coordinate projection $\pi_{1}:\sd HG\units \to \ho$ is a
homeomorphism.

\begin{example}\label{ex-sd-space}
  Suppose that $G$ acts on the left of a \emph{space} $H$.  Then
  $G$ acts on $H$ by isomorphisms when $H$ is viewed as a
  groupoid.  Then the map $(h,x)\mapsto (x,x^{-1}h)$ is a groupoid
  isomorphism of the semidirect product $\sd HG$ onto the
  action groupoid $G\rtimes H$.
\end{example}

\begin{remark}
  Although their notation is slightly different, in \cite{bm:xx16}
  Buss and Meyer introduce semidirect products calling them
  ``transformation groupoids for classical actions of $G$ on $H$''.
  If $G$ and $H$ have Haar systems and $G$ acts on $H$ by
  isomorphisms, then Buss and Meyer show in
  \cite{bm:xx16}*{Theorem~5.1} that $\hsdg$ always has a Haar system.
  Since we want the additional structure of a dynamical system, we
  have to require that, in addition, $H$ have an invariant Haar
  system.  Since this makes constructing a Haar system on the
  semidirect product routine,  we give the
  elementary construction that is required for our
  purposes.
  
For \'etale groupoids, Brownlowe \emph{et al}~\cite{bprrwzappa}
  introduced Zappa-Sz\'ep products of groupoids,
  which generalize our semidirect products in that context.
\end{remark}

\begin{lem}\label{lem-haar-sdp} Suppose that $G$ acts on $H$ by
  isomorphisms and that $H$ has an invariant Haar system
  $\lambdah=\set{\lambdah^{v}}_{v\in\ho}$ so that
  \[
    \int_H f(x\cdot t)\,d\lambdah^v(t)=\int_H f(t)\,d\lambdah^{x\cdot
      v}(t).
  \]
  Then we get a Haar system on the semidirect product $\sd HG$ by
  \[
    d\lambda^{(v,w)}(t,x)=d\lambda_H^v(t)\,d\lambda_G^w(x).
  \]
\end{lem}

\begin{proof}
  It suffices to integrate over $\hsdg$ functions of the form
  $f\otimes g$ for $f\in C_c(H)$ and $g\in C_c(G)$:
  \begin{align*}
    &\int_{\hsdg} (f\otimes g)\bigl((t,x)(u,y)\bigr)\,d\lambda^{s(t,x)}(u,y)
    \\&\quad=\int_G \int_H (f\otimes g)(t(x\cdot
    u),xy)\,d\lambda^{(x\inv\cdot s(t),s(x))}(u,y) 
    \\&\quad=\int_Hf(t(x\cdot u))\,d\lambda_H^{x\inv\cdot s(t)}(u)
    \int_Gg(xy)\,d\lambda^{s(x)}(y)
    \\&\quad=\int_Hf(tu)\,d\lambda_H^{s(t)}(u)
    \int_Gg(y)\,d\lambda^{r(x)}(y)\righttext{(invariance)}
    \\&\quad=\int_Hf(u)\,d\lambda_H^{r(t)}(u)
    \int_Gg(y)\,d\lambda^{r(x)}(y)
    \\&\quad=\int_{\hsdg} (f\otimes g)(u,y)\,d\lambda^{r(t,x)}(u,y).
    \qedhere
  \end{align*}
\end{proof}


Suppose that $p:\AA\to H$ is a Fell bundle and that $G$ acts on $\AA$
by isomorphisms.  We get a Banach bundle over $\hsdg$
via the pull-back
$\pi_{1}^{*}\AA=\set{(a,h,x)\in \AA\times \sd HG:p(a)=h}$.  As
usual, we adopt a more compact notation and write elements of this
pull-back as pairs $(a,x)$ in $\AA\times G$ such that
$\rho(p(a))=r(x)$.  Then we define the Fell-bundle operations by
\begin{align}
  \label{eq:40}
  (a,x)(b,y)=\bigl(a(x\cdot b),xy\bigr)\quad\text{and}\quad (a,x)^{*}=
  (x^{-1}\cdot a^{*},x^{-1}).
\end{align}
We denote the pull-back with these operations by $\sd\AA G$ and write
$p'$ for the projection $p'(a,x)=(p(a),x)$ onto
$\hsdg$.
Then we can generalize the constructions in \cite[Section~6]{kmqw1} as
follows.

\begin{lemma}
  With the operations defined above, $p':\sd \AA G\to \hsdg$
  is a Fell bundle.
\end{lemma}
\begin{proof}[Proof]
  
  We have the candidates for multiplication and involution defined
  above, so let's proceed with checking \defnref{def-fell-bundle} with
  respect to these operations.  It is routine to verify that the
    multiplication is continuous, bilinear, and associative, and the
    involution is continuous.  Notice that $p$ respects all of the
  operations on $\AA$ and is also equivariant for the $G$
  action. Moreover, $G$ acts on $\AA$ by isomorphisms.
  
  For the multiplication condition, see that
  \begin{align*}
    p'\bigl((a,x)(b,y)\bigr)
    &=p'\bigl(a(x\cdot b),xy\bigr)
    \\&=\Bigl(p\bigl(a(x\cdot b)\bigr),xy\Bigr)
    \\&=\bigl(p(a)p(x\cdot b),xy\bigr)
    \\&=\Bigl(p(a)\bigl(x\cdot p(b)\bigr),xy\Bigr)
    \\&=\bigl(p(a),x\bigr)\bigl(p(b),y\bigr)
    \\&=p'(a,x)p'(b,y).
  \end{align*}

  For the compatibility of the projection with involution,
  \begin{align*}
    p'\bigl((a,x)^*\bigr)
    &=p'(x\inv\cdot a^*,x\inv)
    \\&=\bigl(p(x\inv\cdot a^*),x\inv\bigr)
    \\&=\bigl(x\inv\cdot p(a)^*,x\inv\bigr)
    \\&=\bigl(p(a),x\bigr)^*
    \\&=p'(a,x)^*.
  \end{align*}

  Next, we check that the involution on $\sd \AA G$ is, in fact,
  involutive:
  \begin{align*}
    \bigl((a,x)(b,y)\bigr)^*
    &=\bigl(a(x\cdot b),xy\bigr)^*
    \\&=\Bigl((xy)\inv\cdot \bigl(a(x\cdot b)\bigr)^*,(xy)\inv\Bigr)
    \\&=\Bigl(y\inv x\inv\cdot \bigl((x\cdot b)^*a^*\bigr),y\inv x\inv\Bigr)
    \\&=\Bigl(\bigl(y\inv x\inv\cdot (x\cdot b^*)\bigr)
    \bigl(y\inv x\inv\cdot a^*\bigr),y\inv x\inv\Bigr)
    \\&=\bigl((y\inv\cdot b^*)(y\inv x\inv\cdot a^*),y\inv x\inv\bigr)
    \\&=\Bigl((y\inv\cdot b^*)\bigl(y\inv\cdot (x\inv\cdot a^*)\bigr),
    y\inv x\inv\Bigr)
    \\&=(y\inv\cdot b^*,y\inv)(x\inv\cdot a^*,x\inv)
    \\&=(b,y)^*(a,x)^*.
  \end{align*}

  Before turning to the final conditions, notice that the fibre over
  $(h,x)\in \sd HG$ can be identified as a Banach space with
  $A(h)$ via $(a,x)\mapsto a$, and $A(h)$ is, by assumption, an
  $A(r(h))\sme A(s(h))$-\ib\ with
  respect to the operations inherited from $\AA$. Likewise, the fibre
  over a unit $(u,v)\in \sd HG^{(0)}$ is identified with the
  \cs-algebra $A(u)$.  The operations on $A(u,v)$ correspond to
    those of $A(u)$, so $A(u,v)$ is a \cs-algebra.  Similarly, on the
    left-hand side the module and inner product operations on $A(h,x)$
    correspond to those on $A(h)$.  On the right-hand side, there is a
    subtlety: the module and inner product on $A(h,x)$ correspond to
    those on $A(h)$ provided that the $C^*$-algebra $A(s(h,x))$ is
    identified with $A(s(h))$ via the map $(a,s(x))\mapsto x\cdot a$.
    More precisely,
    \begin{align*}
      (a,x)\cdot (b,u)=\bigl(a(x\cdot b),x\bigr),
    \end{align*}
    which, after $(b,u)$ is sent to $x\cdot b$, corresponds to
    $a\cdot c=ac$ for $a\in A(h),c\in A(s(h))$, and
    \begin{align*}
      (a,x)^*(b,x)=\bigl(x\inv\cdot (a^*b),s(x)\bigr),
    \end{align*}
    which would be sent to the element $a^*b\in A(s(h))$.  
  \end{proof}
  
  \begin{rem}
    Duwenig and Li~\cite{dlzappa}*{Theorem~3.8}
  studied Zappa-Sz\'ep products of Fell bundles over \'etale groupoids,
  which generalize our semidirect-product Fell bundles in that context.
  \end{rem}

  \section{Crossed Products}\label{sec:main}

  We want to see that there is a natural groupoid dynamical system
  associated to a groupoid action on a Fell bundle by
  isomorphisms. For this, we require that $H$ has an invariant Haar
  system $\lambdah=\set{\lambda_{H}^{v}}_{v\in \ho} $.  Furthermore,
  if $u\in\go$, then $H_{u}=\set{h\in H:\rho(h)=u}$ is the groupoid
  $\rho^{-1}(u)H\rho^{-1}(u)$ and $\set{\lambdah^{v}}_{\rho(v)=u}$ is
  a Haar system for $H_{u}$.  We fix a Fell bundle $p:\AA\to H$ and an
  action of $G$ by isomorphisms on $\AA$.  Let $\rho_{H}:H\to\go$ be
  the moment map.  If $\phi\in C_{b}(\go)$ and
  $f\in \Gamma_{c}(H,\AA)$, then let
  \begin{align}
    \label{eq:47}
    V(\phi)f(h)=\phi\bigl(\rho_{H}(h)\bigr)f(h).
  \end{align}
  If $f,g\in\Gamma_{c}(H,\AA)$, then
  \begin{align}
    \label{eq:48}
    V(\phi)(f*g)(h)
    &=\phi\bigl(\rho_{H}(h)\bigr) \int_{H}
      f(k)g(k^{-1}h)\,d\lambdah^{r(h)}(k) \\
    \intertext{which, since $\rho_{H}$ is constant on $H_{r(h)}$ by
    Remark~\ref{rem-sat}, is}
    &= \int_{H} f(k) \phi(\rho_{H}(k^{-1}h)g(k^{-1}h)
      \,d\lambdah^{r(h)}(k) \\
    &= f*\bigl(V(\phi)g\bigr)(h).
  \end{align}

  We can view $M(\cs(H,\AA))$ as the collection of adjointable
  operators on $\cs(H,\AA)$ viewed a right Hilbert module over itself
  with the inner product $\langle f,g\rangle=f^{*}*g$
  \cite{RW}*{Definition~2.48}.  Since
  $V(\phi)f^{*}=(V(\bar\phi)f)^{*}$, it follows from
  \begin{align}
    \label{eq:49}
    V(\phi)(f*g)=f*(V(\phi)g)
  \end{align}
  that $\langle V(\phi)f,g\rangle=\langle f,V(\bar \phi)g\rangle$.
  Then the usual arguments (like \cite{wil:toolkit}*{Lemma~1.48} for
  example) show that $V(\phi)$ extends to a bounded adjointable
  operator in the multiplier algebra $M(\cs(H,\AA))$ which lies in the center in view of
  \eqref{eq:49} and \cite{danacrossed}*{Lemma~8.3}.  Thus we have
  proved most of the following.

\begin{lemma}
  Suppose that $p:\AA\to H$ is a Fell bundle and that $G$ acts on
  $\AA$ by isomorphisms.  Then $\cs(H,\AA)$ is a $C_{0}(\go)$-algebra
  with $\cs(H,\AA)(u)=\cs(H_{u},\AA)$.
\end{lemma}
\begin{proof}
  All that is left is to identify the fibre.  But this can be done as
  in \cite{wil:toolkit}*{Proposition~5.37} using
  \cite{simwil:ijm13}*{Lemma~9} in place of
  \cite{wil:toolkit}*{Theorem~5.1}.
(This identification
    requires the universal \cs-norm. We do not know if it holds for
    the reduced norm.)
\end{proof}

Since $\lambdah$ is invariant, for each $x\in G$ we get an isomorphism
\begin{equation}
  \label{eq:54}
  \alpha_{x}:\cs(H_{s(x)},\AA)\to \cs(H_{r(x)},\AA)
\end{equation}
given by
\begin{equation}
  \label{eq:55}
  \alpha_{x}(f)(h)=x\cdot f(x^{-1}\cdot h)\quad\text{for $f\in
    \Gamma_{c}(H_{s(x)},\AA)$. }
\end{equation}

\begin{lemma}
  \label{lem-g-action} Let $p:\AA\to H$ be a Fell bundle.  Suppose
  that $G$ acts on $H$ by isomorphisms and that
  $\lambdah=\set{\lambdah^{v}}_{v\in \ho} $ is an invariant Haar
  system on $H$.  Then $\alpha=\set{\alpha_{x}}_{x\in G}$ as in
  \eqref{eq:55} above implements a groupoid dynamical system
  $(\cs(H,\AA),G,\alpha)$.
\end{lemma}
\begin{proof}[Sketch of the Proof]
  The only issue is continuity of the action.  Since $\cs(H,\AA)$
  is a $C_{0}(\go)$-algebra, it can be realized as the sections of a
  Banach bundle $q:\EE\to\go$ with fibres $E(u)=\cs(H_{u},\AA)$.  We
  get a $G$-action on $\EE$ via $x\cdot e=\alpha_{x}(e)$.  To see that
  this is a continuous action, suppose that $e_{i}\to e$ and
  $x_{i}\to x$ with $q(e_{i})=s(x_{i})$.  To see that
  $x_{i}\cdot e_{i}\to x\cdot e$, we'll use
  \cite{mw:fell}*{Lemma~A.3}.  If $f\in \Gamma_{c}(H,A)$, then we get
  a section $\check f\in \Gamma_{c}(\go,\EE)$ where
  $\check f(u)\in \Gamma_{c}(H_{u},\AA)$ is given by
  \begin{align}
    \label{eq:57}
    \check f(x)(h)=f(h)\quad\text{for $h\in H_{u}$.}
  \end{align}

  Fix $\epsilon>0$.  Then we can choose $f$ so that
  $\|\check f(q(e))-e\|<\epsilon$.  By upper semicontinuity, we can
  assume $\|\check f(q(e_{i}))-e_{i}\|<\epsilon$ for all $i$.  Since
  $\alpha$ is norm reducing, we also have
  $\|\alpha_{x_{i}}(\check f(q(e_{i}))) -
  \alpha_{x_{i}}(e_{i})\|<\epsilon$ for all $i$.  Thus it will suffice
  to see that
  $\alpha_{x_{i}}(\check f(q(e_{i})))\to \alpha_{x}(\check f(q(e)))$
  in the topology on $\EE$.
  
  Our approach is modeled on the proof of \cite{muhwil:nyjm}*{Lemma~4.3}.
  Let $G*_{s}H=\set{(x,h):\rho(p(h))=s(x)}$, and let $\cca$ be the
  collection of continuous functions $F:G*H\to \AA$ such that
  $\rho(p(F(x,h)))=s(x)$ and $F$ vanishes off a compact subset of
  $G*_{s}H$.  Then $F$ determines a function $\hat F$ of $G$ into
  $\EE$ where $\hat F(x)\in\Gamma_{c}(H_{s(x)},\AA)\subset E(s(x))$ is
  given by
  \begin{equation}
    \label{eq:56}
    \hat F(x)(h)=F(x,h).
  \end{equation}
  If $F$ is of the form $(x,h)\mapsto \phi(x)f(h)$ for
  $\phi\in C_{c}(G)$ and $f\in \Gamma_{c}(H,\AA)$, then $\hat F$ is
  continuous from $G$ into $\EE$.  Since sums of such functions are
  appropriately dense in $\cca$, we see that $x\mapsto \hat F(x)$ is
  continuous for all $F\in\cca$.

  We get the result by considering
  $F(y,h)=\phi(y)y\cdot f(y^{-1}\cdot h)$ for $\phi\in C_{c}(G)$
  identically one near $x$.  Then
  $\hat F(y)(h)=\alpha_{y}(\check f(s(y))(h)$ for $y$ near $x$.
\end{proof}

As in \cite{mw:fell}*{Example~2.1}, we can realize
$\cs(H,\AA)\rtimes_{\alpha}G$ as the \cs-algebra of the Fell bundle
$\BB=r^{*}\EE=\set{(b,x)\in\EE\times G:q(b)=r(x)}$ where
  $\cs(H,\AA)\cong \Gamma_{0}(\go,\EE)$ as a above.  The multiplication
is given by $(b,x)(b',y)=(b\alpha_{x}(b'),xy)$ and involution by
$(b,x)^{*}= (\alpha_{x}^{-1}(b),x^{-1})$.  If
$a\in \Gamma_{c}(H,\AA)$, then we can define
$b_{a}\in \Gamma_{c}(\go,\EE)$ by $b_{a}(u)$ where
$b_{a}(u)\in \Gamma_{c}(H_{u},\AA)$ and $b_{a}(u)(h)=a(h)$.  If
$\phi\in C_{c}(G)$, then we get $(a\otimes \phi)\in \Gamma_{c}(G,\BB)$
where $(a\otimes\phi)(x)=\bigl(\phi(x)b_{a}(r(x)),x\bigr)$.  To keep
the notation from obscuring what is going on, we usually simply write
$(a\tensor\phi)(x)=\bigl(\phi(x)a(\cdot),x\bigr)$.  It is not hard to
see that such sections span a dense subspace of $\Gamma_{c}(G,\BB)$ in
the inductive limit topology.  Note that
\begin{align}
  \label{eq:72}
  (a\tensor\phi)(x)(a'\tensor\phi')(y)
  &=\bigl(\phi(x) a(\cdot),x\bigr)
    \bigl(\phi'(y) a'(\cdot),y\bigr) \\
  &=\bigl(\phi(x)\phi'(y) a(\cdot) \alpha_{x}(a'(\cdot)),xy\bigr) \\
  &=\bigl(\phi(x)\phi'(y) a''(\cdot), xy\bigr),
\end{align}
where
\begin{align}
  \label{eq:73}
  a''(h)=\int_{H} a(k) x\cdot a'(x^{-1}\cdot (k^{-1}h))
  \,d\lambdah^{r(h)}(k) \quad\text{for $h\in H_{r(xy)}=H_{r(x)}$.}
\end{align}

More generally, every $f\in\Gamma_{c}(\sd HG,\sd\AA G)$ is determined
by a function $F_{f}:\sd HG\to\AA$ where $f(h,g)=(F_{f}(h,x),x)$ and
$\rho(f(h,x))=r(x)$.  Then we get a section
$\check f\in \Gamma_{c}(G,\BB)$ given by
$\check f(x)=\bigl(F_{f}(\cdot,x),x\bigr)$ where $F_{f}(\cdot ,x)$ is
the corresponding element of $ \Gamma_{c}(H_{r(x)},\AA)$ as above:
$F_{f}(\cdot,x)(h)=F_{f}(h,x)$.  Then
\begin{align}
  \label{eq:41}
  \check f*\check g(x)
  &=\int_{G} \check f(y) \check
    g(y^{-1}x)\,d\lambdag^{r(x)}(y) \\
  &=\int_{G}\bigl(F_{f}(\cdot,y),y\bigr)
    \bigl(F_{g}(\cdot,y^{-1}x),y^{-1}x\bigr) \,d\lambdag^{r(x)}(y) \\
  &=\int_{G} \bigl(F_{f}(\cdot ,y)\alpha_{y}\bigl(F_{g}(\cdot,
    y^{-1}x)\bigr),x\bigr) \,d\lambdag^{r(x)}(y) \\
  &=\int_{G}\bigl(a'(\cdot),x\bigr)\,d\lambdag^{r(x)}(y)
\end{align}
where $a'\in \Gamma_{c}(H_{r(x)},\AA)$ is given by
\begin{align}
  \label{eq:74}
  a'(h)=\int_{H} F_{f}(k,y) y\cdot F_{g}(y^{-1}\cdot
  (k^{-1}h),y^{-1}x) \,d\lambdah^{r(h)}(k).
\end{align}
Arguing as in \cite{danacrossed}*{Lemma~1.108}, we see that
$\check f*\check g(x)=(a''(\cdot),x)$ with
$a''\in \Gamma_{c}(H_{r(x)},\AA)$ and
\begin{align}
  \label{eq:61}
  a''(h)
  &= \int_{G} \int_{H} F_{f}(k,y)y\cdot F_{g}(y^{-1}\cdot
    (k^{-1}h),y^{-1}x) \, d\lambdah^{r(h)}(k) \,d\lambda_{G}^{r(x)}(y).
\end{align}

On the other hand, supposing
$f,g\in\Gamma_{c}\bigl(\sd HG,\sd\AA G\bigr)$, we have
\begin{align}
  \label{eq:75}
  f*g(h,x)
  &=\int_{G} \int_{H} f(k,y) g\bigl((k,y)^{-1}(h,x)\bigr)
    \,d\lambdah^{r(h)} (k) \,d\lambdag^{r(x)}(y) \\
  &=\int_{G} \int_{H} f(k,y) g\bigl((y^{-1}\cdot k^{-1},y^{-1})(h,x)\bigr)
    \,d\lambdah^{r(h)} (k) \,d\lambdag^{r(x)}(y) \\
  &=\int_{G} \int_{H} f(k,y) g
    (y^{-1}\cdot (k^{-1}h), y^{-1}x)
    \,d\lambdah^{r(h)} (k) \,d\lambdag^{r(x)}(y) \\
  &=\int_{G} \int_{H}
    \bigl(F_{f}(k,y),y\bigr) \bigl(F_{g}(y^{-1}\cdot
    (k^{-1}h),y^{-1}x),y^{-1}x\bigr)  \,d\lambdah^{r(h)} (k)
    \,d\lambdag^{r(x)}(y)  \\
  &= \int_{G} \int_{H}
    \bigl(F_{f}(k,y) y\cdot F_{g}(y^{-1}\cdot (k^{-1}h),y^{-1}x),x
    \bigr) \,d\lambdah^{r(h)} (k)
    \,d\lambdag^{r(x)}(y).
\end{align}

\begin{thm}
  \label{thm-cross-prod} Suppose that $p:\AA\to H$ is a Fell bundle
  and that $G$ acts on $\AA$ by automorphisms.  We suppose that the
  Haar system on $H$ is $G$-invariant so that $\sd HG$ has a Haar
  system as in Lemma~\ref{lem-haar-sdp}. Then $\cs(\sd HG,\sd\AA G)$
  is isomorphic to the crossed product $\cs(H,\AA)\rtimes_{\alpha}G$
  for the dynamical system from Lemma~\ref{lem-g-action}.
\end{thm}
  
By the above, the map $f\mapsto \check f$ from
$\Gamma_{c}(\sd HG,\sd\AA G)$ to $\Gamma_{c} (G,\BB)$ sends $f*g$ to
$\check f*\check g$.  A similar computation shows that it sends
$f^{*}$ to $(\check f)^{*}$.  Thus we get a $*$-homomorphism
$\Phi: \Gamma_{c}(\sd HG,\sd\AA G) \to \Gamma_{c} (G,\BB)$.  It is not
hard to see that this map is continuous in the inductive limit
topology and has dense range.  Hence $\Phi$ extends to a surjective
homomorphism of $\cs(\sd HG,\sd\AA G)$ onto
$\cs(G,\BB)\cong \cs(H,\AA)\rtimes_{\alpha} G$.

To show that $\Phi$ is isometric, and therefore the desired
isomorphism, we will proceed as follows.  Given a faithful
representation, $L$, of $\cs(\sd HG,\sd\AA G)$, we will produce a
representation $\underline L$ of $\cs(G,\BB)$ such that
$L(f)=\underline L(\Phi(f))$.  Then
$\|f\|=\|L(f)\|=\|\underline L(\Phi(f))\|\le \|\Phi(f)\|\le \|f\|$.
This will suffice.

We will identify $\sd HG\units$ with $\ho$.  (Then
$s(h,x)=x^{-1}\cdot s_{H}(h)$ and $r(h,x)=r_{H}(h)$.)  By
\cite{mw:fell}*{Theorem~4.13}, we can assume that $L$ is the
integrated form of a strict representation $(\mu,\ho*\HH,\pi)$.  We
can assume that $\mu$ is finite and use
\cite{danacrossed}*{Theorem~I.5} to disintegrate the quasi-invariant
measure $\mu$ with respect to the moment map $\rho:\ho\to\go$.  Thus
off a $\mu$-null set $N(\mu)$ we obtain probability measures $\mu^{u}$
on $\ho$ with support in $H_{u}\units=(H_{u})\units$ such that for
every suitable Borel function $h$ on $\ho$
\begin{align}
  \label{eq:79}
  \int_{\ho}h(v)\,d\mu(v)=\int_{\go}\int_{\ho} h(v) \,d\mu^{u}(v) \,
  d\mu_{G}(u) 
\end{align}
where $\mu_{G}=\rho_{*}\mu$ is the push-forward
  of $\mu$
to $\go$
(push-forward
  measures are examined in \cite{danacrossed}*{Lemma~H.13}).
  We let $\Delta$ be the modular function on $\sd HG$ for
$\mu$.  Since there are a number of groupoids that will support
  representations, we will have multiple quasi-invariant measures and
  their modular functions in play.  Although we will always assume
  modular functions on our groupoids are taken to be Borel
  homomorphisms into the positive multiplicative reals
  (\cite{wil:toolkit}*{Proposition~7.6}), one of the technical
  difficulties will be to insure all these functions play nice with
  one-another.  In part
  because modular functions are only defined almost everywhere, this
  will be a delicate operation culminating in
  Proposition~\ref{prop-formula-Delta} where it is crucial that the
  equality proved there holds everywhere.

Let $\nu_{H}=\mu\circ \lambdah$ and
$\nu_{G}=\mu_{G}\circ\lambdag$.  For a summary of some of the measure
theory required, see \cite{wil:toolkit}*{\S3.1}.

\begin{lemma}\label{lem-mug-quasi}
  The push forward $\mu_{G}=\rho_{*}\mu$ is quasi-invariant on $\go$.
  Thus if $\Delta_{G}$ is the corresponding modular function on $G$,
  we have
  \begin{align}
    \label{eq:96}
    \int_{\go}\int_{G}f(x^{-1})\Delta_{G}(x^{-1}) \,d\lambdag^{u}(x)
    \,d\mu_{G}(u) = \int_{\go}\int_{G}f(x) \,d\lambdag^{u}(x)
    \,d\mu_{G}(x) 
  \end{align}
  for all $f\in C_{c}(G)$.
\end{lemma}
\begin{proof}
  Suppose that $f$ is a non-negative Borel function on $G$ and that
  $g\in C^{+}(H)$ is a nowhere vanishing function such that
  $\lambdah(g)(v)=1$ for all $u\in\ho$ (see
  \cite{wil:toolkit}*{Lemma~3.19}).  Then
  \begin{align}
    \label{eq:77}
    \nu_{G}(f)
    &= \int_{\go}\int_{G}f(x)\,d\lambdag^{u}(x) \,d\mu_{G}(u) \\
    &= \int_{\ho}\int_{G}f(x) \,d\lambdag^{\rho(v)}(x) \,d\mu(v) \\
    &= \int_{\ho}\int_{G}\int_{H} f(x) g(h)  \,d\lambdah^{v}(h)
      \,d\lambdag^{\rho(v)} (x) \,d\mu(v) \\
    &=\int_{\ho}\int_{G}\int_{H} f(x^{-1}) g(x^{-1}\cdot h^{-1})
      \Delta(x^{-1}\cdot h^{-1},x^{-1}) \,d\lambdah^{v}(h)
      \,d\lambdag^{\rho(v)} (x) \,d\mu(v) \\
    &=\int_{\ho}\int_{G}f(x^{-1}) \underbrace{\Bigl( \int_{H} g(x^{-1}\cdot h^{-1})
      \Delta(x^{-1}\cdot h^{-1},x^{-1}) \,d\lambdah^{v}(h)\Bigr)}_{B(x,v)} \\
      &\hskip3in
      \,d\lambdag^{\rho(v)}(x) 
      \,d\mu(v) \\
    &=\int_{\ho}\int_{G}f(x^{-1})B(x,v) \,d\lambdag^{\rho(v)}(x)
      \,d\mu(v) \\
    &=\int_{\go}\int_{\ho}\int_{G} f(x^{-1}) B(x,v)
      \,d\lambdag^{\rho(v)}(x) \,d\mu^{u}(v) \,d\mu_{G}(u) \\
    &=\int_{\go}\int_{\ho}\int_{G} f(x^{-1}) B(x,u)
      \,d\lambdag^{u}(x) \,d\mu^{u}(v) \,d\mu_{G}(u) \\
    &=\int_{\go}\int_{G} f(x^{-1}) \underbrace{\Bigl( \int_{\ho}B(x,v)
      \,d\mu^{u}(v)\Bigr)}_{B(x)} \,d\lambdag^{u}(x) \,d\mu_{G}(u)  \\
    &=\int_{\go}\int_{G}f(x^{-1}) B(x) \,d\lambdag^{u}(x) \,d\mu_{G}(u)
    \\
    &= \nu_{G}^{-1}(f B^{*}).
  \end{align}

  Since the $g$ and $\Delta$ are both strictly positive, $B$ is a
  Borel function taking values in $(0,\infty]$.  Thus if $f$ is the
  characteristic function of a Borel set in $G$, $f$ is equal to zero
  $\nu_{G}$-almost everywhere and and only if it is equal to zero
  $\nu_{G}^{-1}$-almost everywhere.  Hence $\nu_{G}$ and $\nu_{G}^{-1}$
  are equivalent measures on $G$. That is, $\mu_{G}$ is
  $G$-quasi-invariant.
\end{proof}

Let $\grho=\set{(v,x)\in\sd HG:v\in\ho}$.  Then $\grho$ is a closed
subgroupoid of $\sd HG$ with unit space identified with $\ho$.  Thus
$s(v,x)=x^{-1}\cdot v$ and $r(v,x)=v$.  Then
$(v,x)(x^{-1}\cdot v,y)=(v,xy)$.
We can view $\grho$ as the
  action groupoid for the $G$ action on $\go$ where we have used a
  different convention for writing the action groupoid as pairs---see
  Remark~\ref{rem-other}.
  As in Lemma~\ref{lem-haar-sdp}, we get a
Haar system $\sigma$ on $\grho$ where
\begin{align}
  \sigma(f)(v)=\int_{G}f(v,x)\,d\lambdag^{\rho(v)}(x).
\end{align}

\begin{lemma}
  \label{lem-phi-equi} The measure $\mu$ is quasi-invariant when $\ho$
  is viewed as the unit space of $\grho$.  Thus if we write $\delta$
  for its modular function on $\grho$, then $\delta$ is a Borel
  homomorphism such that
  \begin{align}
    \label{eq:97}
    &\int_{\ho}\int_{G} \phi(x^{-1}\cdot v,x^{-1})\delta(x^{-1}\cdot v,
    x^{-1}) \,d\lambdag^{\rho(v)} \,d\mu(v)\\
    & = \int_{\ho}\int_{G}
    \phi(v,x) \,d\lambdag^{\rho(v)}(x) \,d\mu(v)
  \end{align}
  for suitable Borel functions $\phi$ on $\grho$.
\end{lemma}
\begin{proof}
  The proof is similar to that for Lemma~\ref{lem-mug-quasi}.  In
  particular, let $g\in C^{+}(H)$ be nowhere vanishing with
  $\lambdah(g)(v)=1$ for all $v\in\ho$.  Let
  $\nugrho= \sigma\circ \nu$.  Then
  \begin{align}
    \label{eq:87}
    \nugrho^{-1}(f)\hskip-1cm&\\
                             &= \int_{\ho}\int_{G} f(x^{-1}\cdot v,x^{-1})
                               \,d\lambdag^{\rho(v)}(x)\,d\mu(v) \\
                             &= \int_{\ho}\int_{G}\int_{H} \underbrace{f(x^{-1}\cdot
                               r_{H}(h),x^{-1})g(h)}_{F(h,x)} 
                               \,d\lambdah^{v}(h) \,d\lambdag^{\rho(v)}(x) \,d\mu(v) \\
                             &=\int_{\ho}\int_{G}\int_{H}F(h,x)  \,d\lambdah^{v}(h)
                               \,d\lambdag^{\rho(v)}(x) \,d\mu(v) \\
                             &=\int_{\ho}\int_{G}\int_{H} F(x^{-1}\cdot h^{-1},x^{-1})
                               \Delta(x^{-1}\cdot h^{-1},x^{-1})   \,d\lambdah^{v}(h)
                               \,d\lambdag^{\rho(v)}(x) \,d\mu(v) \\
                             &= \int_{\ho}\int_{G}f(v,x)\underbrace{\Bigl(\int_{H}  b(x^{-1}\cdot h^{-1})
                               \Delta(x^{-1}\cdot h^{-1},x^{-1})   \,d\lambdah^{v}(h)  \Bigr)}_{B(x)}\\
                               &\hskip3in
                               \,d\lambdag^{\rho(v)}(x) \,d\mu(v) \\
                             &=\nugrho(f B).
  \end{align}

  Thus $\nugrho$ and $\nugrho^{-1}$ are equivalent and $\nu$ is
  $\grho$-quasi-invariant.  As usual, we can then choose the modular
  function $\delta$ to be a Borel homomorphism.
\end{proof}

\begin{lemma}
  \label{lem-mu-h-quasi} The measure $\mu$ is also quasi-invariant
  when $\ho$ is viewed as the unit space of $H$.
\end{lemma}
\begin{proof}
  Let $f$ be a non-negative Borel function and then use
  \cite{wil:toolkit}*{Proposition~3.19} to find $b\in C^{+}(G)$ such
  that $\lambdag(b)(u)=1$ for all $u\in\go$ and such that $b$ never
  vanishes.  Then
  \begin{align}
    \label{eq:53}
    \nu_{H}(f)
    &=\int_{\ho}\int_{H} f(h) \,d\lambdah^{v}(h) \,d\mu(v) \\
    &=\int_{\ho}\int_{G}\int_{H}f(h) b(x)
      \,d\lambdah^{v}(h)\,d\lambdag^{\rho(v)}(x) \,d\mu(v) \\
    \intertext{which, since $\mu$ is $\sd HG$-quasi-invariant, is}
    &=\int_{\ho}\int_{G}\int_{H} f(x^{-1}\cdot h^{-1}) b(x^{-1})
      \Delta(x^{-1}\cdot h^{-1},x^{-1}) \,d\lambdah^{v}(h)
      \,d\lambdag^{\rho(v)}(x) \,d\mu(v) \\
    \intertext{which, since $x^{-1}\cdot h^{-1}=(x^{-1}\cdot h)^{-1}$
    and since $\lambdah$ is invariant, is}
    &=\int_{\ho}\int_{G}\underbrace{\Bigl( \int_{H} f(h^{-1}) b(x^{-1})
      \Delta(h^{-1},x^{-1}) \,d\lambdah^{x^{-1}\cdot v}(h)
      \Bigr)}_{C(x^{-1}\cdot v,x^{-1})}
      \,d\lambdag^{\rho(v)}(x) \,d\mu(v) \\
    &= \int_{\ho}\int_{G} C(v,x) \delta(v,x)^{-1} \,d\lambdag^{\rho(v)}(x)
      \,d\mu(v) \\
    &= \int_{\ho}\int_{G} \int_{H} f(h^{-1})b(x) \Delta(h^{-1}, x)
      \delta(v,x) \,d\lambdah^{v}(h)  \,d\lambdag^{\rho(v)}(x)
      \,d\mu(v) \\
    &= \int_{\ho}\int_{H} f(h^{-1}) \underbrace{\Bigl( \int_{G}b(x)\Delta(h^{-1},x)
      \delta(v,x)^{-1} \,d\lambdag^{\rho(v)}(x) \Bigr)}_{D(h^{-1})}
      \,d\lambdah^{v}(h) 
      \,d\mu(v) \\
    &= \int_{\ho}\int_{H}f(h^{-1})D(h^{-1}) \,d\lambdah^{v}(h)
      \,d\mu(v) =\nu_{H}^{-1}(fD).
  \end{align}
  The result follows as $D$ takes values in $(0,\infty]$.
\end{proof}

In view of Lemma~\ref{lem-mu-h-quasi}, we can write $\Delta_{H}$ for
the modular function on $H$ when $\mu$ is viewed as a quasi-invariant
measure on $\ho\subset H$.  Recall that we are using $\delta$ for the
modular function on $\grho$ from Lemma~\ref{lem-phi-equi}.

\begin{prop}
  \label{prop-formula-Delta} We can take
  $\Delta_{H}(h)=\Delta(h,\rho(h))$ and $\delta(v,x)=\Delta(v,x)$.
  Then for all $(h,x)\in \sd HG$, we have
  \begin{align}
    \label{eq:118}
    \Delta(h,x)=\Delta_{H}(x)\delta(s_{H}(h),x).
  \end{align}
\end{prop}

For the proof, we need the following observation.

\begin{lemma}
  \label{lem-tech-ae} Suppose that $f$ and $\phi$ are bounded
  non-negative Borel functions on $H$ and $\grho$, respectively, and
  that
  \begin{align}
    \label{eq:83}
    f(h)=\phi(s(h),x)\quad\text{for $\nu$-almost all $(h,x)\in\sd HG$.}
  \end{align}
  Then there is a non-negative bounded Borel function on $A$ on $\ho$
  such that
  \begin{align}
    \label{eq:84}
    f(h)=A(s(h))\quad\text{for $\nu_{H}$-almost all $h\in H$.}
  \end{align}
\end{lemma}
\begin{proof}
  Let $b\in C_{c,r}^{+}(G)$ be such that $\lambdag(b)(u)=1$ for all
  $u\in\go$.  Then let
  \begin{align}
    \label{eq:134a}
    A(v)=\int_{G} \phi(v,x) b(x)\,d\lambdag^{\rho(v)} (x).
  \end{align}
  Let $g\in C_{c}(H)$.  Then
  \begin{align}
    \label{eq:76a}
    \int_{\ho}\int_{H}
    &\bigl(f(h)-A(s(h))\bigr)g(h) \,d\lambdah^{v}(h)
      \,d\mu(v) \\
    &= \int_{\ho}\int_{H}\int_{G} f(h)g(h)b(x)
      \,d\lambdag^{\rho(v)}(x) \,d\lambdah^{v}(h) \,\mu(v) \\
    &\hskip1in - \int_{\ho}
      \int_{H} \phi(s(h),x)b(x) g(h)\,d\lambdag^{\rho(s(h))}(x)
      \,d\lambdah^{v}(h) \,d\mu(v) \\
    \intertext{which, since $r(h)=v$ implies $\rho(s(h))=\rho(v)$, is}
    &= \int_{\ho}\int_{G}\int_{H} \bigl(f(h)-\phi(s(h),x)\bigr) b(h) g(h)
      \,d\lambdah^{v}(h)\,d\lambdag^{\rho(v)}(x)\,d\mu(v) =0
  \end{align}

  Since $g\in C_{c}(H)$ is arbitrary, the result follows.
\end{proof}

\begin{proof}[Proof of Proposition~\ref{prop-formula-Delta}]
  We have
  \begin{align}
    \label{eq:98}
    \nu(f)
    &=\int_{\ho}\int_{G}\underbrace{\Bigl(\int_{H}f(h,x)
      \,d\lambdah^{v}(h)\Bigr)}_{\phi(v,x)} 
      \,d\lambdag^{\rho(v)} (x) \,d\mu(v) \\
    \intertext{which, in view of Lemma~\ref{lem-phi-equi}, is}
    &= \int_{\ho}\int_{G} \int_{H} f(h,x^{-1})
      \,d\lambdah^{x^{-1}\cdot v}(h)  \delta(x^{-1}\cdot v,x^{-1})
      \,d\lambdag^{\rho(v)} (x) 
      \,d\mu(v) \\
    \intertext{which, by invariance, is}
    &= \int_{\ho}\int_{G} \int_{H} f(x^{-1}\cdot
      h,x^{-1})\delta(x^{-1}\cdot v, x^{-1})
      \,d\lambdah^{v}(h)  \,d\lambdag^{\rho(v)} (x) 
      \,d\mu(v) \\
    &=\int_{\ho}\int_{H}\underbrace{\Bigl( \int_{G} f(x^{-1}\cdot
      h,x^{-1})\delta(x^{-1}\cdot r_{H}(h) ,x^{-1})
      \,d\lambdag^{\rho(v)}(x)\Bigr)}_{E(h)} \,d\lambdah^{v}(h)
      \,d\mu(v) \\
    \intertext{which, since $\mu$ is $H$-quasi-invariant, is}
    &= \int_{\ho}\int_{H}\int_{G} f(x^{-1}\cdot h^{-1},x^{-1})
      \delta(x^{-1}\cdot s_{H}(h),x^{-1}) \,d\lambdag^{\rho(v)}(x)
      \Delta_{H}(h^{-1}) \\
      &\hskip3in 
      \,d\lambdah^{v}(h) \,d\mu(v) \\
    &= \int_{\ho}\int_{G}\int_{H} f(x^{-1}\cdot h^{-1},x^{-1})
      \Delta_{H}(h^{-1}) \delta(x^{-1}\cdot s_{H}(h),x^{-1}) \\
    &\hskip3in
      \,d\lambdah^{v}(h) \,d\lambdag^{\rho(v)}(x) \,d\mu(v).
  \end{align}
  Therefore
  \begin{align}
    \label{eq:99}
    \Delta(x^{-1}\cdot
    h^{-1},x^{-1})=\Delta_{H}(h^{-1})\delta(x^{-1}\cdot s_{H}(h),x^{-1})
  \end{align}
  $\nu$-almost everywhere.  Since $\Delta$, $\Delta_{H}$, and $\delta$
  are homomorphisms, we have
  \begin{align}
    \label{eq:100}
    \Delta(h,x)=\Delta_{H}(h)\delta(s_{H}(h),x)
  \end{align}
  $\nu$-almost everywhere as claimed.

  However, if $(h,x)\in\sd HG$, then $\rho(h)=r(x)$ and $(h,k)$ may be written as 
  $(h,x)=(h,r(x))(s_{H}(h),x)=(h,\rho(h))(s_{H}(h),x)$.  Since
  $\Delta$ is a homomorphism
  \begin{align}
    \label{eq:114}
    \Delta(h,x)=\Delta(h,\rho(h))\Delta(s_{H}(h),x).
  \end{align}
  Therefore it will suffice to see that
  $ \Delta_{H}(h)=\Delta(\rho(h),h)$ off a $\nu_{H}$-null set that
  $\delta( v,x)=\Delta(v,x)$ off a $\nu_{\grho}$-null set.

  It follows from \eqref{eq:114} that
  \begin{align}
    \label{eq:85}
    \Delta_{H}(h)\Delta(h,\rho(h))^{-1} = \Delta(s_{H}(h),x)\delta(s_{H}(h),x)^{-1}
  \end{align}
  for $\nu$-almost all $(h,x)$.  Therefore Lemma~\ref{lem-tech-ae}
  implies that there is a bounded Borel function $A$ on $\ho$ such
  that
  \begin{align}
    \label{eq:86}
    \Delta_{H}(h)=\Delta(h,\rho(h))A(s(h)) \quad\text{$\nu_{H}$-almost
    everywhere.} 
  \end{align}

  Therefore $\Delta'_{H}(h)=\Delta(h,\rho(h))$ is a modular function
  on $H$ for $\mu$ and hence must be multiplicative for
  $\nu_{H}^{(2)}$-almost all $(h,k)\in H^{(2)}$ by
  \cite{wil:toolkit}*{Lemma~7.5}.  It follows that $A(s(h))=1$ for
  $\nu_{H}^{(2)}$-almost all $(h,k)\in H^{(2)}$.  But then for any
  $g\in C_{c}(H)$ and an appropriate $b\in C^{+}(G)$, we have
  \begin{align}
    \label{eq:78}
    \int_{\ho}\int_{H}& \bigl(A(s(h)-1\bigr)g(h)
                        \,d(\lambda_{H})_{v}(h)\,d\mu(h) \\
                      &= \int_{\ho}\int_{H}\int_{H}\Bigl(A(s(h))-1\bigr) g(h)b(k)
                        \,d(\lambda_{H})_{v}(h) \,d\lambdah^{v}(k) \,d\mu(v) =0.
  \end{align}
  Therefore $A(s(h))=1$ $\nu_{H}^{-1}$-almost everywhere, and hence
  $\nu_{H}$-almost everywhere.

  Hence we immediately have $\Delta(h,\rho(h))=\Delta_{H}(h)$
  $\nu_{H}$-almost everywhere as required.

  Now let $f\in C_{c}(\sd HG)$.  Then
  \begin{align}
    \label{eq:88}
    \int_{\ho}&\int_{G}\int_{H}
                f(h,x)
                \,d\lambdah^{v}(h)\,d\lambdag^{\rho(v)} (x)\,d\mu(v) \\
              &= \int_{\ho}\int_{G}\Bigl( \int_{H} f(x^{-1}\cdot h^{-1},x^{-1})
                \Delta(x^{-1}\cdot h^{-1},x^{-1}) \,d\lambdah^{v}(h) \Bigr)
                \,d\lambdag^{\rho(v)} (x) \,d\mu(v) \\
              &= \int_{\ho}\int_{G} \int_{H} f(h^{-1},x^{-1})
                \Delta(h^{-1}, x^{-1}) 
                \,d\lambdah^{x^{-1}\cdot v}(h) 
                \,d\lambdag^{\rho(v)} (x) \,d\mu(v) \\
              &= \int_{\ho}\int_{G} \underbrace{\Bigl(\int_{H} f(h^{-1},x^{-1})
                \Delta(h^{-1}, x^{-1}) 
                \,d\lambdah^{x^{-1}\cdot v}(h) \Bigr)}_{C(x^{-1}\cdot v,x^{-1})}
                \,d\lambdag^{\rho(v)} (x) \,d\mu(v)   \\
              &=\int_{\ho}\int_{G}C(v,x)\delta(v,x)^{-1} \,d\lambdag^{\rho(v)} (x)
                \,d\mu(v) \\
              &= \int_{\ho} \int_{G} \int_{H} f(h^{-1},x) \Delta(h^{-1},x)
                \delta(v,x)^{-1} \,d\lambdah^{v}(h) \,d\lambdag^{\rho(v)}(x)
                \,d\mu(v) \\
              &= \int_{\ho} \int_{G} \int_{H} f(h^{-1},x)
                \Delta(h^{-1},\rho(h^{-1})) \Delta(v,x)
                \delta(v,x)^{-1} \\
                &\hskip3in \,d\lambdah^{v}(h) \,d\lambdag^{\rho(v)}(x)
                \,d\mu(v) \\
    \intertext{which, since $\Delta_{H}'(h)=\Delta(h,\rho(h))$ is a
    modular function for $\mu$, is}
              &=\int_{\ho}\int_{G}\int_{H} f(h,x) 
                \,d\lambdah^{v}(h) \Delta(v,x)\delta(x,v)^{-1}
                \,d\lambdag^{\rho(v)}(x) \,d\mu(v).
  \end{align}
  Since $f\in C_{c}(\sd HG)$ is arbitrary, it follows that
  $\Delta(v,x)=\delta(v,x)$ for $\nu_{\grho}$-almost all $(v,x)$ as
  required.

  This completes the proof.
\end{proof}

From this point forward we will assume
$\Delta_{H}(h)=\Delta(h,\rho(h))$ and $\delta(v,x)=\Delta(v,x)$, and
hence that $\Delta(h,x)=\Delta_{H}(h)\delta(s(h),x)$ for all
$(h,x)\in\sd HG$.

Since $\mu_{G}$ is a quasi-invariant measure on $\go$, whenever
$U\subset\go$ is $\mu_{G}$-conull
the reduction
$G|_{U}=\set{x\in G:\text{$r(x)\in U$ and $s(x)\in U$}}$ is
$\nu_{G}$-conull (see \cite{wil:toolkit}*{Remark~7.8}).  Muhly has
called such reductions \emph{inessential}.

\begin{prop}\label{prop-key}
  There is a $\mu_{G}$-conull set $U\subset \go$ such that for all
  $x\in G|_{U}$ we have
  $\mu^{s(x)}=\delta(\cdot, x^{-1}) \Delta_{G}(x) (x^{-1}\cdot
  \mu^{r(x)})$.  That is, for all suitable Borel functions $\phi$ on
  $\ho$ and $x$ in the $\nu_{G}$-conull reduction $G|_{U}$, we have
  \begin{align}
    \label{eq:101}
    \int_{\ho} \phi(v) \,d\mu^{s(x)}(v) =\int_{\ho} \phi(x^{-1}\cdot
    v) \delta(x^{-1}\cdot v,x^{-1})\Delta_{G}(x) \,d\mu^{r(x)}(v).
  \end{align}
\end{prop}
\begin{proof}
  Suppose $\phi\in C_{c}(\ho)$ and $\psi\in C_{c}(G)$.  Then
  \begin{align}
    \label{eq:93}
    \nugrho(\phi\tensor\psi)
    &=\int_{\ho}\int_{G} \phi(v) \psi(x) \,d\lambdag^{\rho(v)}(x)
      \,d\mu(v) \\
    &= \int_{\go}\int_{\ho}\int_{G}\phi(v) \psi(x) \,d\lambdag^{u}(x)
      \,d\mu^{u}(v) \,d\mu_{G}(u) \\
    &= \int_{\go}\int_{G} \Bigl( \int_{\ho}\phi(v) \,d\mu^{u}(v)\Bigr)
      \psi(x) \,d\lambdag^{u}(x) \,d\mu_{G}(u) \\
    &=\int_{G}\Bigl(\int_{\ho}\phi(v)\,d\mu^{r(x)}(v) \Bigr) \psi(x)
      \,d\nu_{G}(x). 
  \end{align}

  On the other hand, using Lemma~\ref{lem-phi-equi},
  \begin{align}
    \label{eq:94}
    \nugrho(\phi\tensor&\psi)
                         = \int_{\ho}\int_{G} \phi(x^{-1}\cdot v)\psi(x^{-1})
                         \delta(x^{-1}\cdot v,x^{-1}) \,d\lambdag^{\rho(v)} (x) \,d\mu(v)
    \\
                       &=\int_{\go}\int_{G}\underbrace{\Bigl( \int_{\ho} \phi(x^{-1}\cdot v)
                         \delta(x^{-1}\cdot v,x^{-1}) \,d\mu^{r(x)}(v) \Bigr) \psi(x^{-1})}_{D(x)}\\
                         &\hskip3in
                         \,d\lambdag^{u}(x) \,d\mu_{G}(u) \\
    \intertext{which, since $\mu_{G}$ is quasi-invariant, is}
                       &= \int_{\go}\int_{G} D(x^{-1}) \Delta_{G}(x^{-1})
                         \,d\lambdag^{u}(x) \,d\mu_{G}(u) \\
                       &= \int_{\go}\int_{G}\Bigl( \int_{\ho} \phi(x\cdot v)
                         \delta(x\cdot v,x) \Delta_{G}(x^{-1})\,d\mu^{s(x)}(v) \Bigr) \psi(x)\\
                         &\hskip3in
                         \,d\lambdag^{u}(x) \,d\mu_{G}(u) \\
                       &= \int_{\go}\int_{G}\Bigl( \int_{\ho} \phi(v)
                         \delta(v,x) \Delta_{G}(x^{-1})\\
                        &hskip3in
                         \,d(x\cdot \mu^{s(x)})(v) \Bigr) \psi(x)
                         \,d\lambdag^{u}(x) \,d\mu_{G}(u) \\
                       &= \int_{G} \Bigl( \int_{\ho} \phi(v)
                         \delta(v,x) \Delta_{G}(x^{-1})\,d(x\cdot \mu^{s(x)})(v)\Bigr)
                         \psi(x) \,d\nu_{G}(x).
  \end{align}
  Since $\psi\in C_{c}(G)$ is arbitrary, there is a $\nu_{G}$-null set
  $N(\phi)$ such that
  \begin{multline}
    \label{eq:95}
    \int_{\ho}\phi(v) \,d\mu^{r(x)}(v) \\= \int_{\ho}\phi(v)
    \delta(v,x)\Delta_{G}(x^{-1}) \,d(x\cdot
    \mu^{s(x)})(v) \quad\text{if $x\notin N(\phi)$.}
  \end{multline}
  Since $C_{c}(\ho)$ is separable in the inductive limit topology
  (\cite{wil:toolkit}*{Lemma~C.2}), there is a $\nu_{G}$-null set $N$
  such that \eqref{eq:95} holds for all $\phi\in C_{c}(\ho)$.  Since
  $\nu_{G}$ and $\nu_{G}^{-1}$ are equivalent, we get \eqref{eq:101}.

  Let
  $\Sigma=\set{x\in G:\mu^{s(x)}=\delta(\cdot, x^{-1}) \Delta_{G}(x)
    (x^{-1}\cdot \mu^{r(x)})}$.  If $x,y\in \Sigma$ and
  $(y,x)\in G^{(2)}$, then
  \begin{align}
    \label{eq:102}
    \int_{\ho}\phi(v)\,d\mu^{s(x)}(v)
    &= \int_{\ho} \phi(x^{-1}\cdot v) \delta(x^{-1}\cdot v,x^{-1})
      \Delta_{G}(x) \,d\mu^{r(x)}(v) \\
    \intertext{which, since $r(x)=s(y)$, is}
    &= \int_{\ho} \phi(x^{-1}\cdot y^{-1}\cdot v) \delta(x^{-1}\cdot
      y^{-1}\cdot v,x^{-1}) \delta(y^{-1}\cdot v,y^{-1}) \\
    &\hskip2in 
      \Delta_{G}(y)\Delta_{G}(x) \,d\mu^{r(y)}(v) \\
    \intertext{which, since $\delta$ and $\Delta_{H}$ are homomorphisms,
    is}
    &= \int_{\ho}\phi((yx)^{-1}\cdot v) \delta((yx)^{-1}\cdot v,(yx)^{-1})
      \Delta_{H}(xy) \, d\mu^{r(yx)}(v).
  \end{align}
  That is, $yx\in\Sigma$.  It now follows from
  \cite{wil:toolkit}*{Proposition~D.2} that there is an inessential
  reduction $G|_{U}\subset \Sigma$.
\end{proof}

As usual, we let $H_{u}$ be the subgroupoid $\rho^{-1}(u)$ equipped
with the Haar system
$\lambda_{H_{u}}=\set{\lambdah^{v}}_{v\in \ho_{u}}$ where
$\ho_{u}=\ho\cap H_{u}$.

\begin{lemma}
  \label{lem-one-more} For $\mu_{G}$-almost all $u$, $\mu^{u}$ is
  quasi-invariant on $H_{u}^{(0)}$ and $\Delta_{H}$ restricts to a
  modular function for $\mu^{u}$ on $H_{u}$.
\end{lemma}
\begin{proof}
  Let $\nu_{H_{u}}=\mu^{u}\circ \lambda_{H_{u}}$.  Then if
  $f\in C_{c}(H)$ and $\phi\in C_{c}(\go)$ we have
  \begin{align}
    \label{eq:103}
    \int_{\ho}\int_{H}f(h)\phi(\rho(h)
    & \,d\lambdah^{v}(h) \,d\mu(v)
    \\
    &= \int_{\go}\int_{\ho}\int_{H_{u}} f(h) \,d\lambda_{H_{u}}^{v}(h)
      \,d\mu^{u}(v) \phi(u) \,d\mu_{G}(u) \\
    &= \int_{\go} \nu_{H_{u}}(f) \phi(u)\,d\mu_{G}(u).\\
    \intertext{On the other hand,}
    \int_{\ho}\int_{H}f(h)\phi(\rho(h))
    & \,d\lambdah^{v}(h) \,d\mu(v)\\
    &= \int_{\ho}\int_{H} f(h^{-1})\Delta_{H}(h^{-1}) \,d\lambdah^{v}(h)
      \phi(u) \,d\mu(v) \\
    &= \int_{\go}\int_{\ho} \int_{H} f(h^{-1})\Delta_{H}(h^{-1})
      \,d\lambda_{H_{u}}^{v} (h) \,d\mu^{u}(v) \phi(u)\,d\mu_{G}(u) \\
    &= \int_{\go} \nu_{H_{u}}^{-1}(\Delta_{H}f) \phi(u)\,d\mu_{G}(u).
  \end{align}
  Since $\phi$ is arbitrary, there is a $\mu_{G}$-null set $N(f)$ such
  that $\nu_{H_{u}}(f) = \nu_{H_{u}}^{-1}(\Delta_{H}f)$ if
  $u\notin N(f)$.  Since $C_{c}(H)$ is separable in the inductive
  limit topology \cite{wil:toolkit}*{Proposition~C.3}, there is a
  $\mu_{G}$-null set $N$ such that
  $\mu_{H_{u}}(f)=\nu_{H_{u}}(\Delta_{H} f)$ for all $f\in C_{c}(H)$
  provided $u\notin N$.  The result follows.
\end{proof}

Now we can finish the proof of Theorem~\ref{thm-cross-prod}.

\begin{proof}[Proof of Theorem~\ref{thm-cross-prod}]
  We let $L$ be the integrated form of $(\mu,\ho*\HH,\pi)$ as above.
  Thus if $f\in\Gamma_{c}(\sd HG,\sd\AA G)$ and
  $\xi,\eta\in L^{2}(\ho*\HH,\mu)$, then
  \begin{align}
    \label{eq:80}
    \bigl( L(f)\xi\mid \eta)
    &= \int_{\ho} \int_{G} \int_{H} \bigl(\pi(f(h,x))\xi(x^{-1}\cdot
      s_{H}(h)) \mid
      \eta(v)\bigr)\Delta(h,x)^{-\frac12} \\
    &\hskip2in d\lambdah^{v}(h)
      \,d\lambdag^{\rho(v)}(x) \,d\mu(v) \\
    &= \int_{\go}\int_{\ho} \int_{G} \int_{H} \bigl(\pi(f(h,x))\xi(x^{-1}\cdot
      s_{H}(h)) \mid
      \eta(v)\bigr)\Delta(h,x)^{-\frac12} \\
    &\hskip2in d\lambdah^{v}(h)
      \,d\lambdag^{u}(x) \,d\mu^{u}(v) \,d\mu_{G}(u) \\
    &= \int_{\go}\int_{G} \int_{\ho}\int_{H_{u}} \bigl(\pi(f(h,x))\xi(x^{-1}\cdot
      s_{H}(h)) \mid
      \eta(v)\bigr)\Delta(h,x)^{-\frac12} \\
    &\hskip2in \lambda_{H_{u}}^{v}(h) \,d\mu^{u}(v) \,\lambdag^{u}(x)
      \,d\mu_{G}(u) .
  \end{align}

  Recall that using Proposition~\ref{prop-formula-Delta} we have
  $\Delta(h,x)=\Delta_{H}(h)\delta(s_{H}(h),x)$ for all
  $(h,x)\in\sd HG$.  We may as well also replace $G$ by $G|_{U}$ and
  assume that Proposition~\ref{prop-key} holds for all $x\in G$.  We
  can also shrink $U$ a bit if necessary and assume that
  Lemma~\ref{lem-one-more} holds for all $u\in\go$ as well.

  Let $x\in G$ and $b\in \Gamma_{c}(H_{s(x)}, \AA)$.  Then we define
  $\bar \pi_{x}(b): L^{2}(\ho,\mu^{s(x)})\to L^{2}(\ho,\mu^{r(x)})$ by
  \begin{align}
    \label{eq:81}
    \bigl(&\bar\pi_{x}(b)\xi\mid
    \eta\bigr) \\
      &= \int_{\ho}\int_{H} \bigl(\pi(b(h),x)\xi(x^{-1}\cdot
      s_{H}(h)),\eta(v)\bigr)
      \Delta(h,x)^{-\half}\Delta_{G}(x)^{-\half}\\
    &\hskip2.5in d\lambda_{H_{r(x)}}^{v}(h)
      \,d\mu^{r(x)}(v) \\
    & = \int_{\ho}\int_{H} \bigl(\pi(b(h),x)\xi(x^{-1}\cdot
      s_{H}(h)),\eta(v)\bigr)
      \Delta_{H}(h)^{-\half}\delta(s_{H}(h),x)^{-\half} \\
    &\hskip2.5in 
      \Delta_{G}(x)^{\half}\,d\lambda_{H_{r(x)}}^{v}(h)
      \,d\mu^{r(x)}(v) 
  \end{align}
  Then
  \begin{align}
    \label{eq:82}
    \bigl| \bigl(\bar\pi_{x}(b)\xi\mid \eta\bigr)\bigr|^{2}
    \hskip-1cm&\\
              &\le \Bigl(\int_{\ho}\int_{H}\|b(h)\|\|\xi(x^{-1}\cdot s_{H}(h)) \|
                \|\eta(v)\| \\
              &\hskip1in \Delta_{H}(h)^{-\half}\delta(s_{H}(h),x)^{-\half}
                \Delta_{G}(x)^{\half}\,d\lambda_{H_{r(x)}}^{v}(h)
                \,d\mu^{r(x)}(v) \Bigr)^{2}\\
    \intertext{which, using the Cauchy-Schwarz inequality, is}
              &\le \Bigl( \int_{\ho}\int_{H} \|b(h)\|\|\xi(x^{-1}\cdot
                s_{H}(h))\|^{2} \\
              &\hskip1.4in \Delta_{H}(h)^{-1}\delta(s_{H}(h),x)^{-1}
                \Delta_{G}(x)\,d\lambda_{H_{r(x)}}^{v}(h)
                \,d\mu^{r(x)}(v)\Bigr) \\
              &\hskip1cm \times \Bigl( \int_{\ho}\int_{H}
                \|f(h,x)\|\|\eta(v)\|^{2}\,d\lambda_{H_{r(x)}}^{v}(h)
                \,d\mu^{r(x)}(v) \Bigr) \\
    \intertext{which, using
    Lemma~\ref{lem-one-more}, is}
              &\le \Bigl( \int_{\ho}\|\xi(x^{-1}\cdot
                v)\|^{2}\int_{H} \|b(h)\|  \,d(\lambda_{H_{r(x)}})_{v}(h)
                \\
              &\hskip2in
                \delta(x^{-1}\cdot v,x^{-1})\Delta_{G}(x) \,d\mu^{r(x)}(v)\Bigr) \\
              &\hskip1cm \times \Bigl( \int_{\ho}\int_{H}
                \|b(h)\|\|\eta(v)\|^{2}\,d\lambda_{H_{r(x)}}^{v}(h)
                \,d\mu^{r(x)}(v) \Bigr)  \\
              &\le \Bigl( \|b\|_{I,s} \int_{\ho} \|\xi(x^{-1}\cdot v)\|^{2}
                \delta(x^{-1}\cdot v,x^{-1})\Delta_{G}(x)\,d\mu^{r(x)}(v) \Bigr)
    \\
              &\hskip1cm \times \Bigl(
                \|b\|_{I,r}\int_{\ho}\|\eta(v)\|^{2}\,d\mu^{r(x)}(v) \Bigr)  \\
    \intertext{which, by Proposition~\ref{prop-key}, is}
              &=\|b\|_{I}^{2}\|\xi\|^{2}_{L^{2}(\ho,\mu^{s(x)})}
                \|\eta\|^{2}_{L^{2}(\ho,\mu^{r(x)})}  .
  \end{align}

  Therefore we can extend $\bar \pi_{x}$ to all of $\cs(H_{s(x)},\AA)$
  and define $\Pi:\BB\to \End(\ho*\HH)$ by
  $\Pi(b,x)=\bar\pi_{x}(b)$. We can view $L^{2}(\ho*\HH,\mu)$ as a
  direct integral $L^{2}(\go*\KK,\mu_{G})$ with fibres
  $K(u)=L^{2}(\ho*\HH,\mu^{u})$ as in
  \cite{danacrossed}*{Example~F.19}.  In particular, if
  $\xi\in L^{2}(\ho*\HH,\mu)$, then $\xi$ defines a Borel section
  $\underline \xi \in L^{2}(\go*\KK,\mu_{G})$ with
  $\underline \xi(u)=\xi$.  Thus
  \begin{align}
    \label{eq:105}
    \bigl(L(f)\xi\mid\eta\bigr)
    &= \int_{\go}\int_{G} \bigl(\Pi(\check
      f(x))\underline\xi(s(x))\mid\underline
      \eta(u)\bigr)\Delta_{G}(x)^{-\half} \,d\lambdag^{u}(x)
      \,d\mu_{G}(u). 
  \end{align}

  Therefore it will suffice to see that
  $\hat\Pi((b,x))=(r(x),\bar\pi_{x}(b),s(x))$ is a Borel $*$-functor
  as in \cite{mw:fell}*{Definition~4.5}.  Then we can let
  $\underline L= (\mu_{G},\go*\KK,\Pi)$.

  Linearity is clear.  To check multiplicativity, let
  $\check f(x)=(b(\cdot),x)$ and $\check g(y)= (c(\cdot),y)$ with
  $(x,y)\in G^{(2)}$.  Then
  $\check f(x)\check g(y)=(d(\cdot)\alpha_{x}(c(\cdot)),xy)$ where
  $d(\cdot):= b(\cdot)\alpha_{x}\bigl(c(\cdot)\bigr) \in
  \Gamma_{c}(H_{r(xy)},\AA)$ is given by
  \begin{align}
    \label{eq:106}
    d(h)=\int_{H} b(k)x\cdot c(x^{-1}\cdot(k^{-1}h))
    \,d\lambdah^{r(h)}(k)\quad\text{for $h\in H_{r(x)}$.}
  \end{align}

  Thus using a vector-valued version of \eqref{eq:81},
  \begin{align}
    \label{eq:109}
    \bar\pi_{x}
    &(b)(\bar\pi_{y}(c)\xi)(v) \\
    &=\int_{H}\pi(b(h),x)(\bar\pi_{y}(c)\xi) (x^{-1}\cdot s(h))
      \Delta(h,x)^{-\half} \Delta_{G}(x)^{\half}
      \,d\lambda_{H}^{v}(h) \\
    &=\int_{H} \int_{H} \pi(b(h),x) \pi(c(k), y)\xi(y^{-1}\cdot s(k))
      \Delta(k,y)^{-\half}
      \Delta_{G}(y)^{\half} \\
    &\hskip1in \Delta(h,x)^{-\half}
      \Delta_{G}(x)^{\half} \,d\lambdah^{x^{-1}\cdot s(h)}(k) \,d
      \lambdah^{v}(h) \\
    &=\int_{H} \int_{H} \pi(b(h),x) \pi(c(x^{-1}\cdot k),
      y)\xi(y^{-1}\cdot x^{-1}\cdot s(k)) \Delta(x^{-1}\cdot k,y)^{-\half}
      \Delta_{G}(y)^{\half} \\
    &\hskip1in \Delta(h,x)^{-\half}
      \Delta_{G}(x)^{\half} \,d\lambdah^{s(h)}(k) \,d
      \lambdah^{v}(h) \\
    \intertext{which, after $k\mapsto h^{-1}k$, is}
    &=\int_{H} \int_{H} \pi(b(h),x) \pi(c(x^{-1}\cdot (h^{-1}k)),
      y)\xi(y^{-1}\cdot x^{-1}\cdot s(k))
    \\
    &\hskip1in \Delta(x^{-1}\cdot(h^{-1}k),y)^{-\half}
      \Delta_{G}(y)^{\half} \\
    &\hskip1in \Delta(h,x)^{-\half}
      \Delta_{G}(x)^{\half} \,d\lambdah^{v}(k) \,d
      \lambdah^{v}(h) 
    \\
    \intertext{which, since $\Delta$ and $\Delta_{G}$ are
    homomorphism, is}
    &=\int_{H} \int_{H} \pi(b(h),x) \pi(c(x^{-1}\cdot (h^{-1}k)),
      y)\xi(y^{-1}\cdot x^{-1}\cdot s(k))
    \\
    &\hskip1in \Delta(k,xy)^{-\half}\Delta_{G}(xy)^{\half}
      \,d\lambdah^{v}(k) \,d 
      \lambdah^{v}(h) 
    \\
    &=\int_{H}\Bigl(\int_{H} \pi\bigl(b(h)x\cdot c(x^{-1}\cdot
      (h^{-1}k)\bigr) \,d\lambdah^{v}(h) \Bigr) \xi(y^{-1}\cdot
      x^{-1}\cdot s(k))
    \\
    &\hskip1in \Delta(k,xy)^{-\half}
      \Delta_{G}(xy) ^{\half} \,d\lambdah^{v}(k) \\
    &=\int_{H} \pi(d(h),xy) \xi((xy)^{-1}\cdot s(h))
      \Delta(k,xy)^{-\half}
      \Delta_{G}(xy)^{-\half}
      \,d\lambdah^{v}(k) \\
    &=\bar\pi_{xy}(d)\xi(v).
  \end{align}
Thus $\Pi(b,x)\Pi(c,y)=\Pi((b,x)(c,y))$ as required.

Another computation shows that $\Pi((b,x)^{*})=\Pi((b,x))^{*}$.  

Since $\pi$ is Borel, it is not hard to check that $\Pi$ is as well.
This completes the proof.
\end{proof}


\begin{bibdiv}
\begin{biblist}

\bib{AR}{book}{
      author={Anantharaman-Delaroche, C.},
      author={Renault, J.},
       title={Amenable groupoids},
      series={Monographies de L'Enseignement Math{\'e}matique [Monographs of
  L'Enseignement Math{\'e}matique]},
   publisher={L'Enseignement Math{\'e}matique},
     address={Geneva},
        date={2000},
      volume={36},
        note={With a foreword by Georges Skandalis and Appendix B by E.
  Germain},
}

\bib{bprrwzappa}{article}{
      author={Brownlowe, N.},
      author={Pask, D.},
      author={Ramagge, J.},
      author={Robertson, D.},
      author={Whittaker, M.~F.},
       title={Zappa-{S}z\'{e}p product groupoids and {$C^*$}-blends},
        date={2017},
     journal={Semigroup Forum},
      volume={94},
      number={3},
       pages={500\ndash 519},
}

\bib{bm:xx16}{article}{
      author={Buss, A.},
      author={Meyer, R.},
      title={Iterated crossed products for groupoid fibrations},
      pages={preprint},
        date={2016},
        note={(arXiv:1604.02015 [math.OA])},
}

\bib{dlzappa}{article}{
      author={Duwenig, A.},
      author={Li, B.},
       title={{Zappa-Sz\'ep product of Fell bundle}},
       date={2020},
       pages = {preprint},
        note={(arXiv:2006.03128v2) [math.OA])},
}

\bib{effhah:mams67}{book}{
      author={Effros, E.~G.},
      author={Hahn, F.},
       title={Locally compact transformation groups and {$C^*$}-algebras},
      series={Memoirs of the American Mathematical Society, No. 75},
   publisher={American Mathematical Society},
     address={Providence, R.I.},
        date={1967},
      review={\MR{37 \#2895}},
}

\bib{felldoran}{book}{
      author={Fell, J. M.~G.},
      author={Doran, R.~S.},
       title={Representations of {$*$}-algebras, locally compact groups, and
  {B}anach {$*$}-algebraic bundles. {V}ol. 1},
      series={Pure and Applied Mathematics},
   publisher={Academic Press Inc.},
     address={Boston, MA},
        date={1988},
      volume={125},
        ISBN={0-12-252721-6},
        note={Basic representation theory of groups and algebras},
      review={\MR{90c:46001}},
}

\bib{goe:thesis}{thesis}{
      author={Goehle, G.},
       title={Groupoid crossed products},
        type={{Ph.D.} Dissertation},
        date={2009},
        note={(arXiv:0905.4681 [math.OA])},
      }

\bib{hkqwstab}{article}{
      author={Hall, L.},
      author={Kaliszewski, S.},
      author={Quigg, J.},
      author={Williams, D.~P.},
       title={{Groupoid semidirect product Fell bundles II ---
           principal actions stablization}},
       date = {2021},
        note={(arXiv:2105.02280  [math.OA])},
}

\bib{husemoller}{book}{
      author={Husemoller, D.},
       title={Fibre bundles},
     edition={Third},
      series={Graduate Texts in Mathematics},
   publisher={Springer-Verlag},
     address={New York},
        date={1994},
      volume={20},
}

\bib{kmqw1}{article}{
      author={Kaliszewski, S.},
      author={Muhly, P.~S.},
      author={Quigg, J.},
      author={Williams, D.~P.},
       title={Coactions and {F}ell bundles},
        date={2010},
     journal={New York J. Math.},
      volume={16},
       pages={315\ndash 359},
}

\bib{kmqw2}{article}{
      author={Kaliszewski, S.},
      author={Muhly, P.~S.},
      author={Quigg, J.},
      author={Williams, D.~P.},
       title={{Fell bundles and imprimitivity theorems}},
        date={2013},
     journal={M\"unster J. Math.},
      volume={6},
       pages={53\ndash 83},
}

\bib{mw:fell}{article}{
      author={Muhly, P.~S.},
      author={Williams, D.~P.},
       title={Equivalence and disintegration theorems for {F}ell bundles and
  their ${C}^*$-algebras},
        date={2008},
     journal={Dissertationes Mathematicae},
      volume={456},
       pages={1\ndash 57},
}

\bib{muhwil:nyjm}{book}{
      author={Muhly, P.~S.},
      author={Williams, D.~P.},
       title={Renault's equivalence theorem for groupoid crossed products},
      series={NYJM Monographs},
   publisher={State University of New York at Albany},
     address={Albany, NY},
        date={2008},
      volume={3},
        note={Available at http://nyjm.albany.edu:8000/m/2008/3.htm},
}

\bib{RW}{book}{
      author={Raeburn, I.},
      author={Williams, D.~P.},
       title={{Morita equivalence and continuous-trace $C^*$-algebras}},
      series={Mathematical Surveys and Monographs},
   publisher={American Mathematical Society},
     address={Providence, RI},
        date={1998},
      volume={60},
}

\bib{simwil:ijm13}{article}{
      author={Sims, A.},
      author={Williams, D.~P.},
       title={Amenability for {F}ell bundles over groupoids},
        date={2013},
        ISSN={0019-2082},
     journal={Illinois J. Math.},
      volume={57},
      number={2},
       pages={429\ndash 444},
         url={http://projecteuclid.org/euclid.ijm/1408453589},
      review={\MR{3263040}},
}

\bib{danacrossed}{book}{
      author={Williams, D.~P.},
       title={Crossed products of {$C{\sp \ast}$}-algebras},
      series={Mathematical Surveys and Monographs},
   publisher={American Mathematical Society},
     address={Providence, RI},
        date={2007},
      volume={134},
}

\bib{wil:toolkit}{book}{
      author={Williams, D.~P.},
       title={A tool kit for groupoid {\cs}-algebras},
      series={Mathematical Surveys and Monographs},
   publisher={American Mathematical Society},
     address={Providence, RI},
        date={2019},
      volume={241},
}

\end{biblist}
\end{bibdiv}


\end{document}